\documentclass{amsart}
\usepackage{amssymb}
\usepackage{amsmath}
\usepackage{amsfonts}

\newtheorem{theorem}{Theorem}
\theoremstyle{plain}

\newtheorem{corollary}{Corollary}

\newtheorem{definition}{Definition}

\newtheorem{lemma}{Lemma}

\newtheorem{proposition}{Proposition}
\newtheorem{remark}{Remark}

\numberwithin{equation}{section}
\input{tcilatex}

\begin{document}
\title[Koszul Gorenstein algebras]{Dualities in Koszul graded AS Gorenstein
algebras}
\author{Roberto Mart\'{\i}nez-Villa}
\address{Centro de Ciencias Matem\'{a}ticas, UNAM, Morelia}
\email{mvilla@matmor.unam.mx}
\urladdr{http://www.matmor.unam.mx}
\date{October 22, 2012}
\dedicatory{}
\thanks{}

\begin{abstract}
The paper is dedicated to the study of certain non commutative graded AS
Gorenstein algebras $\Lambda $ [10] , [13], [14].

The main result of the paper is that for Koszul algebras $\Lambda $ with
Yoneda algebra $\Gamma $, such that both $\Lambda $ and $\Gamma $ are graded
AS Gorenstein noetherian of finite local cohomology dimension on both sides,
there are dualities of triangulated categories:

\underline{$gr$}$_{\Lambda }[\Omega ^{-1}]$ $\cong D^{b}(Qgr_{\Gamma })$ and 
\underline{$gr$}$_{\Gamma }[\Omega ^{-1}]$ $\cong D^{b}(Qgr_{\Lambda })$

where, and $Qgr_{\Gamma }$ is the category of tails, this is: the category
of finitely generated graded modules $gr_{\Gamma }$ divided by the modules
of finite length, and $D^{b}(Qgr_{\Gamma })$ the corresponding derived
category and \underline{$gr$}$_{\Lambda }[\Omega ^{-1}]$ the stabilization
of the category of finetely generated graded $\Lambda $-modules, module the
finetely generated projective modules.
\end{abstract}

\maketitle

\section{\protect\bigskip Introduction}

The paper is dedicated to the study of certain non commutative graded AS
Gorenstein algebras $\Lambda $ [10] , [13], [14] those which are noetherian
of finite local cohomology dimension on both sides, and Koszul. We proved in
[13] that the Yoneda algebra $\Gamma $ of a Koszul graded AS Gorenstein
algebra is again graded AS Gorenstein. We will assume in addition $\Lambda $
and $\Gamma $ are both noetherian and of finite local cohomology dimension
on both sides.

For such algebras we can generalize the classical Bernstein-Gelfand-Gelfand
[3] theorem, which says that there is an equivalence of triangulated
categories: \underline{$gr$}$_{\Lambda }$ $\cong D^{b}(CohP_{n})$, where 
\underline{$gr$}$_{\Lambda }$ is the stable category of the finitely
generated graded $\Lambda $-modules over the exterior algebra in $n$
variables and $D^{b}(CohP_{n})$ is the derived category of bounded complexes
of coherent sheaves on $n$-dimensional projective space.

This theorem was generalized in [15] and [16] as follows:

Let $\Lambda $ be a finite dimensional Koszul algebra with noetherianYoneda
algebra $\Gamma $. Then there is a duality of triangulated categories: 
\underline{$gr$}$_{\Lambda }[\Omega ^{-1}]$ $\cong D^{b}(Qgr_{\Gamma })$,
where \underline{$gr$}$_{\Lambda }[\Omega ^{-1}]$ is the stabilization of 
\underline{$gr$}$_{\Lambda }$ (in the sense of [Buchweitz], [2]) and $%
Qgr_{\Gamma }$ is the category of tails, this is: the category of finitely
generated graded modules $gr_{\Gamma }$ divided by the modules of finite
length, and $D^{b}(Qgr_{\Gamma })$ the corresponding derived category.

The main result of the paper is that for Koszul algebras $\Lambda $ with
Yoneda algebra $\Gamma $, such that both $\Lambda $ and $\Gamma $ are graded
AS Gorenstein noetherian of finite local cohomology dimension on both sides,
there are dualities of triangulated categories:

\underline{$gr$}$_{\Lambda }[\Omega ^{-1}]$ $\cong D^{b}(Qgr_{\Gamma })$ and 
\underline{$gr$}$_{\Gamma }[\Omega ^{-1}]$ $\cong D^{b}(Qgr_{\Lambda
}).\medskip \bigskip $

{\LARGE Thanks}{\large : I express my gratitude to Jun-ichi Miyachi for his
criticism and some helpful suggestions.}

\section{Castelnovo-Mumford regularity}

This section is dedicated to review the concepts and results developed by P.
J$\varnothing $rgensen in [8], [9] and to check they apply to the algebras
considered in the paper, for completeness we reproduce his proofs here. The
main result is the following:

\begin{theorem}
Let $\Lambda $ be a noetherian Koszul AS Gorenstein algebra of finite local
cohomology dimension. Then for any finitely generated graded module $M$
there is a truncation $M_{\geq k}$ such that $M_{\geq k}[k]$ is Koszul.
\end{theorem}

To prove it we use the line of arguments given in [8] and [9] for connected
graded algebras, checking that they easily extend to positively graded
locally finite algebras $A$ over a field $\Bbbk $. This is $A=\underset{%
i\geq 0}{\oplus }A_{i}$ , where $A_{0}=\Bbbk \times \Bbbk \times ...\times
\Bbbk $ and for each $i\geq 0$ dim$_{\Bbbk }A_{i}<\infty $ .

We use the following notation: Given a complex $Y$ of graded left $\Lambda $%
-modules we will denote by $Y^{\prime }$ the dual complex $Y^{\prime
}=Hom_{\Bbbk }(Y,\Bbbk )$.

Given graded $\Lambda $-modules $Y$, $Z$ the degree zero maps will be
denoted by \linebreak $Hom_{Gr_{\Lambda }}(Y,Z)$, $Z[i]$ is the shift
defined as $Z[i]$ $_{j}=Z_{i+j}$ and $Hom_{\Lambda }(Y$, $Z)=$ $\underset{%
i\in 
\mathbb{Z}
}{\oplus }Hom_{Gr_{\Lambda }}(Y$, $Z[i])$.

\begin{proposition}
Let $A$ be a positively graded $\Bbbk $-algebra, $A^{op}$ the opposite
algebra and $X$, $Y$ complexes, $X\in D^{b}(Gr_{A^{op}})$ and $Y\in
D^{-}(Gr_{A})$. Then $(X\overset{L}{\otimes }_{A}Y)^{\prime
}=RHom(Y,X^{\prime })$.
\end{proposition}

\begin{proof}
Let $F\rightarrow Y$ be a quasi-isomorphism from a complex of free modules $%
F $. Then $X\overset{L}{\otimes }_{A}Y\cong X\otimes _{A}F$ and ($X\otimes
_{A}F)^{n}=\underset{p+q=n}{\oplus }X^{p}\otimes F^{q}$, where $F^{q}=%
\underset{J_{q}}{\oplus }A$, hence, ($X\otimes _{A}F)^{n}=\underset{p+q=n}{%
\oplus }X^{p}\otimes \underset{J_{q}}{\oplus }A=\underset{p+q=n}{\oplus }%
\underset{J_{q}}{\oplus }X^{p}.$

Therefore: $Hom_{\Bbbk }($($X\otimes _{A}F)^{n}$,$\Bbbk )=Hom_{\Bbbk }(%
\underset{p+q=n}{\oplus }\underset{J_{q}}{\oplus }X^{p},\Bbbk )=\underset{%
p+q=n}{\tprod }\underset{J_{q}}{\tprod }Hom_{\Bbbk }(X^{p},\Bbbk )$%
\linebreak $=\underset{q}{\tprod }\underset{J_{q}}{\tprod }Hom_{\Bbbk
}(X^{n-q},\Bbbk )$.

In the other hand, RHom$_{A}$(Y,X$^{\prime }$)$^{-n}$=Hom$^{\circ }$(F,X$%
^{\prime }$)$^{-n}$=$\underset{q}{\tprod }$Hom$_{A}$(F$^{q}$,(X$^{\prime }$)$%
^{q-n}$)\linebreak =$\underset{q}{\tprod }$Hom$_{A}(\underset{J_{q}}{\oplus }
$A,(X$^{\prime }$)$^{q-n}$)=$\underset{q}{\tprod }\underset{J_{q}}{\tprod }$%
(X$^{\prime }$)$^{q-n}$=$\underset{q}{\tprod }\underset{J_{q}}{\tprod }$(X$%
^{n-q}$)$^{\prime }$=(X$\otimes _{A}$F)$^{\prime -n}$.
\end{proof}

Let's recall the definition of local cohomology dimension.

\begin{definition}
Let $A=\underset{i\geq 0}{\oplus }A_{i}$ be a positively graded $\Bbbk $%
-algebra with graded Jacobson radical $\mathfrak{m}=\underset{i\geq 1}{%
\oplus }A_{i}$, define a left exact endo functor $\Gamma _{\mathfrak{m}%
}:Gr_{A}^{+}\rightarrow Gr_{A}^{+}$ in the category of bounded above graded
\ $A$-modules $Gr_{A}^{+}$ ,by $\Gamma _{\mathfrak{m}}(M)=\underset{k}{%
\underrightarrow{\lim }}Hom_{A}(A/A_{\geq k}$, $M)$ . Denote by $\Gamma _{%
\mathfrak{m}}^{n}(-)$ , the $n$-th derived functor. It is clear that $\Gamma
_{\mathfrak{m}}^{n}(M)=\underset{k}{\underrightarrow{\lim }}%
Ext_{A}^{n}(A/A_{\geq k}$, $M)$ . We say that $A$ has finite local
cohomology dimension, if there exist a non negative integer $d$ such that
for all $M\in $ $Gr_{A}^{+}$ and $n\geq d$ , $\Gamma _{\mathfrak{m}%
}^{n}(M)=0 $.
\end{definition}

We refer to [5] IX Corollary 2.4a for the proof of the following:

\begin{lemma}
Let $A$ be a $\Bbbk $-algebra and $I$ an injective $A$- $A$ bimodule. The $I$
is injective both as left and as a right $A$-module.
\end{lemma}

In order to prove next proposition we need the following:

\begin{lemma}
Let $A$ be a positively graded left noetherian $\Bbbk $-algebra of finite
local cohomology dimension on the left and $\{Z_{i}\}_{i\in K}$ a family of $%
\Gamma _{m}$ -acyclic modules. Then $\underset{i\in K}{\oplus }Z_{i}$ is $%
\Gamma _{m}$ -acyclic.
\end{lemma}

\begin{proof}
Let $\{Z_{i}\}_{i\in K}$ be a family of $\Gamma _{m}$ -acyclic modules, this
is: each $Z_{i}$ has an injective resolution:

$0\rightarrow Z_{i}\rightarrow I_{0}^{i}\rightarrow I_{1}^{i}\rightarrow
I_{2}^{i}\rightarrow ...I_{k}^{i}\rightarrow I_{k+1}^{i}\rightarrow ...$
such that $0\rightarrow \Gamma _{m}(I_{0}^{i})\rightarrow \Gamma
_{m}(I_{1}^{i})\rightarrow \Gamma _{m}(I_{2}^{i})\rightarrow ...\Gamma
_{m}(I_{k}^{i})\rightarrow \Gamma _{m}(I_{k+1}^{i})\rightarrow $

has homology zero except at degree zero. Since $A$ is noetherian the exact
sequence:

$0\rightarrow \underset{i\in K}{\oplus }(Z_{i})\rightarrow \underset{i\in K}{%
\oplus (}I_{0}^{i})\rightarrow \underset{i\in K}{\oplus }(I_{1}^{i})%
\rightarrow \underset{i\in K}{\oplus }(I_{2}^{i})\rightarrow ...\underset{%
i\in K}{\oplus }(I_{k}^{i})\rightarrow \underset{i\in K}{\oplus }%
(I_{k+1}^{i})\rightarrow ...$

Is an injective resolution of $\underset{i\in K}{\oplus }$($Z_{i})$ and $%
\Gamma _{m}$( $\underset{i\in K}{\oplus }(I_{k}^{i})$)=$\underset{s}{%
\underrightarrow{\lim }}Hom_{A}(A/A_{\geq s},\underset{i\in K}{\oplus }%
(I_{k}^{i}))$ and $A/A_{\geq s}$ finitely presented (again noetherian) $%
\underset{s}{\underrightarrow{\lim }}Hom_{A}(A/A_{\geq s},\underset{i\in K}{%
\oplus }(I_{k}^{i}))=$\linebreak $\underset{s}{\underrightarrow{\lim }}%
\underset{i\in K}{\oplus }Hom_{A}(A/A_{\geq s},(I_{k}^{i}))=\underset{i\in K}%
{\oplus }\underrightarrow{\lim }Hom_{A}(A/A_{\geq s},(I_{k}^{i}))=\underset{%
i\in K}{\oplus }$ $\Gamma _{m}(I_{k}^{i}).$

In fact: $0\rightarrow \Gamma _{m}(\underset{i\in K}{\oplus }%
(Z_{i}))\rightarrow \Gamma _{m}(\underset{i\in K}{\oplus (}%
I_{0}^{i}))\rightarrow \Gamma _{m}(\underset{i\in K}{\oplus }%
(I_{1}^{i}))\rightarrow \Gamma _{m}(\underset{i\in K}{\oplus }%
(I_{2}^{i}))\rightarrow ...\Gamma _{m}(\underset{i\in K}{\oplus }%
(I_{k}^{i}))\rightarrow \Gamma _{m}\underset{i\in K}{\oplus }(I_{k+1}^{i}))$

is isomorphic to $.0\rightarrow \underset{i\in K}{\oplus }\Gamma
_{m}(Z_{i})\rightarrow \underset{i\in K}{\oplus }\Gamma
_{m}(I_{0}^{i})\rightarrow \underset{i\in K}{\oplus }\Gamma
_{m}((I_{1}^{i})\rightarrow \underset{i\in K}{\oplus }\Gamma
_{m}((I_{2}^{i})\rightarrow ...\underset{i\in K}{\oplus }\Gamma
_{m}(I_{k}^{i})\rightarrow \underset{i\in K}{\oplus }\Gamma
_{m}(I_{k+1}^{i})\rightarrow $

the claim follows.
\end{proof}

\begin{proposition}
Let $A$ be a positively graded left noetherian $\Bbbk $-algebra of finite
local cohomology dimension on the left. Then for any $X\in D^{b}(Gr_{A^{e}})$%
, $Y\in D^{-}(Gr_{A})$, there is an isomorphism $R\Gamma _{\mathfrak{m}}(X$ $%
\overset{L}{\otimes }_{A}Y)\cong R\Gamma _{\mathfrak{m}}(X)$ $\overset{L}{%
\otimes }_{A}Y$.
\end{proposition}

\begin{proof}
The complex $X$ is in $D^{+}$, hence, it has an injective resolution with
objects in $Gr_{A^{e}}$ , $X\rightarrow I$ and $X\in D^{b}(Gr_{A^{e}})$
implies H$^{i}(X)=0$ for almost all $i$.

Assume H$^{i}(X)=0$ for $i>s$ and let $Z=Kerd_{s}$, where $%
d_{s}:I^{s}\rightarrow I^{s+1}$ is the differential. Hence, $0\rightarrow
Z\rightarrow I^{s}\rightarrow I^{s+1}\rightarrow I^{s+2}...\rightarrow
I^{s+k}\rightarrow $ is an injective resolution of $Z$ as $A-A$ bimodule.

Since $A$ has finite local cohomology dimension, there exists an integer $t$
such that $\Gamma _{\mathfrak{m}}^{j}(Z)=0$ for $j>t$. If $Z^{\prime }=\func{%
Im}d_{t}$, $d_{t}:I^{t}\rightarrow I^{t+1}$ is the differential, then $%
\Gamma _{\mathfrak{m}}^{j}(Z^{\prime })=0$ for $j>0$ , this is $Z^{\prime }$
is $\Gamma _{\mathfrak{m}}$-acyclic.

The complex $Q:$ $0\rightarrow I^{0}\rightarrow I^{1}\rightarrow
...I^{t}\rightarrow Z^{\prime }\rightarrow 0$ is a complex $\Gamma _{%
\mathfrak{m}}$-acyclic which is quasi-isomorphic to $I$.

The $\Gamma _{\mathfrak{m}}$-acyclic complexes form an adapted class (See
[7], [19]).

Let $L\rightarrow Y$ be a free resolution of $Y$. Then we have isomorphisms: 
$X$ $\overset{L}{\otimes }_{A}Y\cong X\otimes _{A}L\cong Q\otimes _{A}L$.

The module ($Q\otimes _{A}L)^{n}$ is a direct sum of objects in the complex $%
Q$ and $A$ noetherian implies sums of injective is injective, therefore $%
Q\otimes _{A}L$ is $\Gamma _{\mathfrak{m}}$-acyclic.

It follows $R\Gamma _{\mathfrak{m}}(X$ $\overset{L}{\otimes }_{A}Y\cong
\Gamma _{\mathfrak{m}}(Q\otimes _{A}L).$But we have isomorphisms:

$Hom_{A}(A/A_{\geq k}$, $(Q\otimes _{A}L)^{n})=Hom_{A}(A/A_{\geq k}$, $%
Q^{p}\otimes _{A}\underset{J_{n-p}}{\oplus }A)=\underset{J_{n-p}}{\oplus }%
Hom_{A}(A/A_{\geq k}$, $Q^{p})=Hom_{A}(A/A_{\geq k}$, $Q^{p})\otimes _{A}%
\underset{J_{n-p}}{\oplus }A=Hom_{A}(A/A_{\geq k}$, $Q^{p})\otimes
_{A}L^{n-p}$.

Therefore: $\underset{k}{\underrightarrow{\lim }}Hom_{A}(A/A_{\geq k}$, $%
(Q\otimes _{A}L)^{n})=(\underset{k}{\underrightarrow{\lim }}%
Hom_{A}(A/A_{\geq k}$, $Q^{p}))\otimes _{A}L^{n-p}.$ We are using the fact
that $A$ is noetherian, hence $A/A_{\geq k}$ is finitely presented.

We have proved: $\Gamma _{\mathfrak{m}}(Q\otimes _{A}L)\cong \Gamma _{%
\mathfrak{m}}(Q)\otimes _{A}L$, therefore: $R\Gamma _{\mathfrak{m}}(X$ $%
\overset{L}{\otimes }_{A}Y)\cong R\Gamma _{\mathfrak{m}}(X)$ $\overset{L}{%
\otimes }_{A}Y$.
\end{proof}

The proof of the following lemma was given in [8] and reproduced in [14], we
will not give it here.

\begin{proposition}
Let $\Lambda $ be a positively graded $\Bbbk $-algebra such that the graded
simple have projective resolutions consisting of finitely generated
projective modules, $\mathfrak{m}$ the graded radical of $\Lambda $ and $%
\mathfrak{m}^{op}$ the graded radical of $\Lambda ^{op}$. Then for any
integer $k$, $\Gamma _{\mathfrak{m}}^{k}(\Lambda )=\Gamma _{\mathfrak{m}%
^{op}}^{k}(\Lambda ).$
\end{proposition}

We can prove now the following:

\begin{proposition}
Let $A$ be a positively graded locally finite noetherian $\Bbbk $-algebra of
finite local cohomology dimension on both sides. Let $X$, $Y$ be bounded
complexes of finitely generated graded $A$-modules. Then there exists a
natural isomorphism:

$RHom_{A}(R\Gamma _{\mathfrak{m}}(X)$, $Y)\cong RHom_{A}(X$, $Y)$.
\end{proposition}

\begin{proof}
Letting $Y^{\prime }$ be $Y^{\prime }=Hom_{\Bbbk }(Y$,$\Bbbk )$, there is an
isomorphism $RHom_{A}(R\Gamma _{\mathfrak{m}}(X)$, $Y)\cong RHom_{A}(R\Gamma
_{\mathfrak{m}}(X)$, $Y^{\prime \prime })$.

By Proposition 1, $RHom_{A}(R\Gamma _{\mathfrak{m}^{op}}(A)$, $Y^{\prime
\prime })\cong (Y^{\prime }\overset{L}{\otimes }_{A}R\Gamma _{\mathfrak{m}%
^{op}}(A))^{\prime }$.

By Proposition 2, $Y^{\prime }\overset{L}{\otimes }_{A}R\Gamma _{\mathfrak{m}%
^{op}}(A)\cong R\Gamma _{\mathfrak{m}^{op}}(Y^{\prime }\overset{L}{\otimes }%
_{A}A)\cong R\Gamma _{\mathfrak{m}^{op}}(Y^{\prime })$.

Let $F$ be a free resolution of $Y$, it consists of finitely generated $A$%
-modules. Hence $Y^{\prime }$ consists of finitely cogenerated injective $A$%
-modules, then of torsion modules, and $\Gamma _{\mathfrak{m}%
^{op}}(Y^{\prime })\cong \Gamma _{\mathfrak{m}^{op}}(F^{\prime })=F^{\prime
}\cong Y^{\prime }$.

Therefore: $RHom_{A}(R\Gamma _{\mathfrak{m}^{op}}(A)$, $Y)\cong Y^{\prime
\prime }\cong Y$.

Now, there are isomorphisms:

$RHom_{A}(R\Gamma _{\mathfrak{m}}(X)$, $Y)\cong RHom_{A}(R\Gamma _{\mathfrak{%
m}}(A\overset{L}{\otimes _{A}}X)$, $Y)\cong RHom_{A}(R\Gamma _{\mathfrak{m}%
}(A)\overset{L}{\otimes _{A}}X)$, $Y)\cong RHom_{A}(X$, $RHom(R\Gamma _{%
\mathfrak{m}}(A),Y).$

The last isomorphism is by adjunction and the previous one is by Proposition
2.

By Proposition 3, $RHom_{A}(R\Gamma _{\mathfrak{m}}(X)$, $Y)\cong RHom_{A}(X$%
, $RHom(R\Gamma _{\mathfrak{m}^{op}}(A),Y).$

It follows: $RHom_{A}(R\Gamma _{\mathfrak{m}}(X)$, $Y)\cong RHom_{A}(X$, $Y)$%
.
\end{proof}

Next we have:

\begin{lemma}
For complexes $X\in D^{-}(Gr_{A})$, $Y\in D^{+}(Gr_{A})$, there exists a
spectral sequence $E_{2}^{m,n}=Ext_{A}^{m}(h^{-n}X$, $Y)$ converging to $%
Ext_{A}^{n+m}(X$, $Y)$.
\end{lemma}

\begin{proof}
Let $Y\rightarrow J$ be an injective resolution. The complex $X$ is of the
form:

$X:..\rightarrow X^{-m}\rightarrow ...\rightarrow X^{-k}\rightarrow
X^{-k+1}\rightarrow ...X^{-\ell }\rightarrow 0.$

For each $n$, there is a complex: $Hom_{A}(X$, $J^{n}):$

$0\rightarrow Hom_{A}(X^{-\ell },J^{n})\rightarrow Hom_{A}(X^{-\ell
-1},J^{n})\rightarrow Hom_{A}(X^{-k+1},J^{n})\rightarrow ...$

$Hom_{A}(X^{-m},J^{n})\rightarrow ...$

Since $J^{n}$ is injective, $H^{m}(Hom_{A}(X$, $J^{n}))\cong
Hom_{A}(H^{m}(X) $, $J^{n})$.

If $M^{m,n}=Hom_{A}(X^{-m}$, $J^{n})$ , then $M=(M^{m,n})$ is a complex in
the third quadrant.\newline
$%
\begin{array}{ccccccc}
& 0 &  & 0 &  & 0 &  \\ 
& \downarrow &  & \downarrow &  & \downarrow &  \\ 
\longleftarrow & \text{Hom}_{A}\text{(X}^{-m}\text{,J}^{0}\text{)} & 
\longleftarrow ... & \text{Hom}_{A}\text{(X}^{-\ell -1}\text{,J}^{0}\text{)}
& \longleftarrow & \text{Hom}_{A}\text{(X}^{-\ell }\text{,J}^{0}\text{)} & 
\longleftarrow \text{0} \\ 
& \downarrow &  & \downarrow &  & \downarrow &  \\ 
\longleftarrow & \text{Hom}_{A}\text{(X}^{-m}\text{,J}^{1}\text{)} & 
\longleftarrow ... & \text{Hom}_{A}\text{(X}^{-\ell -1}\text{,J}^{1}\text{)}
& \longleftarrow & \text{Hom}_{A}\text{(X}^{-\ell }\text{,J}^{1}\text{)} & 
\longleftarrow \text{0} \\ 
& \downarrow &  & \downarrow &  & \downarrow &  \\ 
& \underset{.}{\overset{.}{.}} &  & \underset{.}{\overset{.}{.}} &  & 
\underset{.}{\overset{.}{.}} &  \\ 
\longleftarrow & \text{Hom}_{A}\text{(X}^{-m}\text{,J}^{t}\text{)} & 
\longleftarrow ... & \text{Hom}_{A}\text{(X}^{-\ell -1}\text{,J}^{t}\text{)}
& \longleftarrow & \text{Hom}_{A}\text{(X}^{-\ell }\text{,J}^{t}\text{)} & 
\longleftarrow \text{0} \\ 
& \downarrow &  & \downarrow &  & \downarrow & 
\end{array}%
$

Taking first the horizontal homology, then the vertical homology, we obtain
the spectral sequence E$_{2}^{m,n}$=Ext$_{A}^{m}$(h$^{-n}$X,Y) which
converges to the homology of the total complex, which by definition, is Ext$%
_{A}^{n+m}$(X,Y) [24].
\end{proof}

For the next lemma we need to assume either $A$ is Gorenstein or it is of
finite local cohomology dimension.

\begin{lemma}
For X$\in $D$^{-}$(Gr$_{A}$), there is a spectral sequence E$_{2}^{m,n}$=Tor$%
_{-m}^{A}$($\Gamma _{\mathfrak{m}^{op}}^{n}$(A),X) converging to $\Gamma _{%
\mathfrak{m}}^{m+n}$(X).
\end{lemma}

\begin{proof}
By definition, $\Gamma _{\mathfrak{m}}^{m}=h^{m}R\Gamma _{\mathfrak{m}}$.
Let $F$ be a free resolution of $X$.

Then we have a double complex $M^{m,n}=($ $R\Gamma _{\mathfrak{m}%
^{op}}A)^{m}\otimes F^{n}$.

The complex $R\Gamma _{\mathfrak{m}^{op}}A$ is bounded in the Gorenstein
case. If $A$ is of finite local cohomology dimension $R\Gamma _{\mathfrak{m}%
^{op}}A$, can be truncated to a bounded complex of $\Gamma _{\mathfrak{m}%
^{op}}$-acyclic modules.

Taking the second filtration, we obtain a spectral sequence E$_{2}^{m,n}$%
=\linebreak Tor$_{-m}^{A}$($\Gamma _{\mathfrak{m}^{op}}^{n}$(A),X)
converging to the total complex of $M$.

We have isomorphisms $TotM\cong (R\Gamma _{\mathfrak{m}^{op}}A)\overset{L}{%
\otimes }_{A}X\cong (R\Gamma _{\mathfrak{m}}A)\overset{L}{\otimes }%
_{A}X\cong R\Gamma _{\mathfrak{m}}X$.
\end{proof}

\begin{definition}
(Castelnovo-Mumford) A complex $X\in D(Gr_{A})$ is called $p$-regular if $%
\Gamma _{\mathfrak{m}}^{m}(X)_{\geq p+1-m}=0$ for all $m$.

If $X$ is $p$-regular but not $p$-$1$-regular, then we say it has Cohen
Macaulay regularity $p$ and write $CMregX=p$. If $X$is not $p$-regular for
any $p$, the we say $CMregX=\infty $ .

If $X$ is $p$-regular for all $p$, this is $R\Gamma _{\mathfrak{m}}X=0$,
then $CMregX=-\infty $.
\end{definition}

Artin and Schelter introduced in [1] a notion of a non commutative regular
algebra that has been very important. We will use here a generalization of
non commutative Gorenstein that extends the notion of Artin-Schelter
regular. This is a variation of the definition given for connected algebras
in [10].

\begin{definition}
Let $\Bbbk $ be a field and $\Lambda $ a locally finite positively graded $%
\Bbbk $-algebra. Then we say that $\Lambda $ is graded Artin-Schelter
Gorenstein (AS\ Gorenstein, for short) if the following conditions are
satisfied:

There exists a non negative integer $n$, called the graded injective
dimension of $\Lambda $, such that:

i) For all graded simple $S_{i}$ concentrated in degree zero and non
negative integers $j\neq n$, Ext$_{\Lambda }^{j}(S_{i},\Lambda )=0$.

ii) We have an equality Ext$_{\Lambda }^{n}(S_{i},\Lambda )=S_{i}^{\prime
}[-n_{i}]$ , with $S_{i}^{\prime }$ a graded $\Lambda ^{op}$-simple.

iii) For a non negative integer $k\neq n,$ Ext$_{\Lambda
^{op}}^{k}(Ext_{\Lambda }^{n}(S_{i},\Lambda ),\Lambda )=0$ and\linebreak\ Ext%
$_{\Lambda ^{op}}^{n}(Ext_{\Lambda }^{n}(S_{i},\Lambda ),\Lambda )=S_{i}$.
\end{definition}

We need to assume now $A$ is graded AS Gorenstein noetherian of finite local
cohomology dimension. Under this conditions the following was proved in [14].

\begin{theorem}
Let $\Lambda $ be a graded AS Gorenstein algebra of graded injective
dimension $n$ and such that all graded simple modules have projective
resolutions consisting of finitely generated projective modules and assume $%
\Lambda $ has finite local cohomology dimension. Then for any graded left
module $M$ there is a natural isomorphism: $D(\underrightarrow{\lim }%
Ext_{\Lambda }^{i}(\Lambda /\Lambda _{\geq k},M))=Ext_{\Lambda }^{n-i}(M,$ $%
D(\Gamma _{\mathfrak{m}}^{n}(\Lambda ))$ , for $0\leq i\leq n$.
\end{theorem}

Let $D_{fg}^{b}(Gr_{A})$ be the subcategory of $D^{b}(Gr_{A})$ of all
bounded complexes with finitely generated homology.

Let $X\in D_{fg}^{b}(Gr_{A})$ and $X\rightarrow I$ an injective resolution.
Since $X$ is bounded, there is an integer $t$ such that $H^{k}(X)=H^{k}(I)=0$
for $k>t$.

As above, we can truncate $I$ to obtain a complex $I_{>}$ consisting of $%
\Gamma _{\mathfrak{m}}$-acyclic modules, $I_{>}\cong X$ and $I_{>}\in
D_{fg}^{b}(Gr_{A})$.

We want to prove $R\Gamma _{\mathfrak{m}}(X)^{\prime }\in D_{fg}^{b}(Gr_{A})$%
.

$X:$ 0$\rightarrow X_{s_{1}}\overset{d_{1}}{\rightarrow }X_{s_{2}}%
\rightarrow ...X_{s_{\ell -1}}\overset{d_{\ell -1}}{\rightarrow }X_{s_{\ell
}}\rightarrow 0.$

We apply induction on $\ell .$

If $\ell =1$, then $X$ is concentrated in degree $s_{1}$ and $X$ of finitely
generated homology means X is finitely generated and it has a projective
resolution:

...$\rightarrow P_{k}\rightarrow P_{k-1}\rightarrow ...P_{1}\rightarrow
P_{0}\rightarrow X\rightarrow 0$ with each $P_{i}$ finitely generated.

Dualizing with respect to the ring we obtain a complex:

$P^{\ast }:0\rightarrow P_{0}^{\ast }\rightarrow P_{1}^{\ast }\rightarrow
...P_{k}^{\ast }\rightarrow P_{k+1}^{\ast }\rightarrow ...$ with homology $%
H^{i}(P^{\ast })=Ext_{A}^{i}(X,A)$. Since $A^{op}$ is noetherian, each $%
Ext_{A}^{i}(X,A)$ is finitely generated.

But it was proved in Theorem ?, $Ext_{A}^{i}(X,A)\cong D((\underrightarrow{%
\lim }Ext_{A}^{n-i}(A/A_{\geq k},X))=(\Gamma _{\mathfrak{m}%
}^{n-i}(X))^{\prime }$ and $Ext_{A}^{i}(X,A)$ finitely generated, implies $%
R\Gamma _{\mathfrak{m}}(X)^{\prime }\in D_{fg}^{b}(Gr_{A})$.

Let $C$ be $C=$ Coker$d_{\ell -1}=H^{\ell }(X)$ and $B_{\ell }=\func{Im}%
d_{\ell -1}.$

Then there is an exact sequence of complexes:

$%
\begin{array}{ccccccccc}
& 0 &  & 0 &  & 0 &  & 0 &  \\ 
& \downarrow &  & \downarrow &  & \downarrow &  & \downarrow &  \\ 
0\rightarrow & X_{s_{1}} & \rightarrow & X_{s_{2}} & \rightarrow ... & 
X_{s_{\ell -1}} & \overset{d_{\ell -1}}{\rightarrow } & B_{\ell } & 
\rightarrow 0 \\ 
& \downarrow &  & \downarrow &  & \downarrow &  & \downarrow &  \\ 
0\rightarrow & X_{s_{1}} & \rightarrow & X_{s_{2}} & \rightarrow ... & 
X_{s_{\ell -1}} & \rightarrow & X_{s_{\ell }} & \rightarrow 0 \\ 
& \downarrow &  & \downarrow &  & \downarrow &  & \downarrow &  \\ 
& 0 &  & 0 &  & 0 & \rightarrow & C & \rightarrow 0 \\ 
&  &  &  &  &  &  & \downarrow &  \\ 
&  &  &  &  &  &  & 0 & 
\end{array}%
$

The complex:

$Y:$ 0$\rightarrow X_{s_{1}}\overset{d_{1}}{\rightarrow }X_{s_{2}}%
\rightarrow ...X_{s_{\ell -1}}\overset{d_{\ell -1}}{\rightarrow }B_{\ell
}\rightarrow 0$ is quasi- isomorphic to the complex:

0$\rightarrow X_{s_{1}}\overset{d_{1}}{\rightarrow }X_{s_{2}}\rightarrow
...X_{s_{\ell -2}}\overset{d_{\ell -2}}{\rightarrow }Z_{s_{\ell
-1}}\rightarrow 0$ with $Z_{s_{\ell -1}}=Kerd_{\ell -1}$.

By induction hypothesis $R\Gamma _{\mathfrak{m}}(Y)^{\prime }\in
D_{fg}^{b}(Gr_{A}).$

We have a triangle $Y\rightarrow X\rightarrow C\rightarrow Y[1]$ which
induces a triangle:

$R\Gamma _{\mathfrak{m}}(Y)\rightarrow R\Gamma _{\mathfrak{m}}(X)\rightarrow
R\Gamma _{\mathfrak{m}}(C)\rightarrow R\Gamma _{\mathfrak{m}}(Y)[1]$

By the long homology sequence, there is an exact sequence:

$\Gamma _{\mathfrak{m}}^{j-1}(C)\rightarrow \Gamma _{\mathfrak{m}%
}^{j}(Y)\rightarrow \Gamma _{\mathfrak{m}}^{j}(X)\rightarrow \Gamma _{%
\mathfrak{m}}^{j}(C)\rightarrow \Gamma _{\mathfrak{m}}^{j+1}(Y)$

Dualizing with respect to $\Bbbk $, there is an exact sequence:

($\Gamma _{\mathfrak{m}}^{j}(C))^{\prime }\rightarrow (\Gamma _{\mathfrak{m}%
}^{j}(X))^{\prime }\rightarrow (\Gamma _{\mathfrak{m}}^{j}(Y))^{\prime }.$

Using $A$ is noetherian and induction, it follows $(\Gamma _{\mathfrak{m}%
}^{j}(X))^{\prime }$ is finitely generated.

Since for any complex $Z$ and any $i$ there is an isomorphism $%
H^{i}(Z)^{\prime }\cong H^{i}(Z^{\prime })$.

It follows $R\Gamma _{\mathfrak{m}}(X)^{\prime }\in D_{fg}^{b}(Gr_{A})$.

Therefore $R\Gamma _{\mathfrak{m}}(X)$ is a complex with finitely
cogenerated homology and each $\Gamma _{\mathfrak{m}}^{j}(X)$ is finitely
cogenerated hence $CMregX\neq \infty $ and $CMregX\neq -\infty $.

In the graded AS Gorenstein case, there is an integer $n$ such that $\Gamma
_{\mathfrak{m}}^{j}(A)=\Gamma _{\mathfrak{m}^{op}}^{j}(A)=0$ for $j\neq n$ .
According to [14], $I_{n}^{\prime }=\Gamma _{\mathfrak{m}}^{n}(A)=\Gamma _{%
\mathfrak{m}^{op}}^{n}(A)=J_{n}^{\prime }$, where $I_{n}^{\prime }=\oplus
D(P_{j}^{\ast })[-n_{\sigma (j)}]$ and $J_{n}^{\prime }=\oplus
D(P_{j})[-n_{\tau (j)}].$

Since $\sigma $ and $\tau $ are permutations, $I_{n}^{\prime }$ is
cogenerated as left module in the same degrees as $J_{n}^{\prime }$ is
cogenerated as right module and $CMreg(_{A}A)=CMreg(A_{A})$.

\begin{definition}
(Ext-regularity) The complex $X\in D(Gr_{A})$ is $r$-$Ext$-$regular$ if
\linebreak $Ext_{A}^{m}(X,A_{0})_{\leq -r-1-m}=0$ for all $m$.

If $X$ is $r$-$Ext$-$regular$ and is not (r-1)-$Ext$-$regular$ we say $Ext$-$%
regular(X)=r$. If $X$ is not r-$Ext$-$regular$ for any $r$, then $Ext$-$%
regular(X)=\infty $ and if for all $r$ the complex $X$ is r-$Ext$-$regular$,
this is $Ext_{A}(X,A_{0})=0$, then $Ext$-$regular(X)=-\infty $.
\end{definition}

In [15] we gave the following definition.

\begin{definition}
A complex of graded modules over a graded algebra is subdiagonal if for each 
$i$ the $i$th module is generated in degrees at least $i$, provided is not
zero.
\end{definition}

We will make use of the following:

\begin{lemma}
Let $A$ be a locally finite graded noetherian algebra over a field $\Bbbk $
and $X$ a complex in $D_{fg}^{-}(Gr_{A})$. Then $X$ has a projective
resolution $P\rightarrow X$ consisting of finitely generated graded
projective modules such that $P$ is subdiagonal.
\end{lemma}

\begin{proof}
Since $X$ has a graded projective resolution $P$ we may consider $P$ instead
of $X$ and prove that $P=P^{\prime }\oplus P^{\prime \prime }$ where $%
P^{\prime }$ is a subdiagonal complex of finitely generated projective
graded modules and $H^{i}(P^{\prime \prime })=0$ for all $i$.

$P:...\rightarrow P_{n+1}\rightarrow P_{n}\rightarrow P_{n-1}\rightarrow
...P_{1}\rightarrow P_{0}\rightarrow 0$

There is an exact sequence: $0\rightarrow B_{1}\rightarrow P_{0}\rightarrow
C\rightarrow 0$ with $H^{0}(P)=C$ finitely generated.

Since $C$ has a finitely generated projective cover $P_{0}^{\prime }$, there
is an exact commutative diagram:

$%
\begin{array}{ccccccc}
& 0 &  & 0 &  & 0 &  \\ 
& \downarrow &  & \downarrow &  & \downarrow &  \\ 
0\rightarrow & B_{1}^{\prime } & \rightarrow & P_{0}^{\prime } & \rightarrow
& C & \rightarrow 0 \\ 
& \downarrow &  & \downarrow &  & \downarrow &  \\ 
0\rightarrow & B_{1} & \rightarrow & P_{0} & \rightarrow & C & \rightarrow 0
\\ 
& \downarrow &  & \downarrow &  & \downarrow &  \\ 
0\rightarrow & P_{0}^{\prime \prime } & \rightarrow & P_{0}^{\prime \prime }
& \rightarrow & 0 &  \\ 
& \downarrow &  & \downarrow &  &  &  \\ 
& 0 &  & 0 &  &  & 
\end{array}%
$

Hence $B_{1}\cong B_{1}^{\prime }\oplus P_{0}^{\prime \prime }$ and $%
B_{1}^{\prime }$ has a finitely generated projective cover $P_{1}^{\prime }$
and there is an exact sequence: $0\rightarrow Z_{1}^{\prime }\rightarrow
P_{1}^{\prime }\rightarrow B_{1}^{\prime }\rightarrow 0.$

We have an exact commutative diagram:

$%
\begin{array}{ccccccc}
& 0 &  & 0 &  & 0 &  \\ 
& \downarrow &  & \downarrow &  & \downarrow &  \\ 
0\rightarrow & Z_{1}^{\prime } & \rightarrow & P_{1}^{\prime }\oplus
P_{0}^{\prime \prime } & \rightarrow & B_{1}^{\prime }\oplus P_{0}^{\prime
\prime } & \rightarrow 0 \\ 
& \downarrow &  & \downarrow &  & \downarrow &  \\ 
0\rightarrow & Z_{1} & \rightarrow & P_{1} & \rightarrow & B_{1}^{\prime
}\oplus P_{0}^{\prime \prime } & \rightarrow 0 \\ 
& \downarrow &  & \downarrow &  & \downarrow &  \\ 
0\rightarrow & P_{1}^{\prime \prime } & \rightarrow & P_{1}^{\prime \prime }
& \rightarrow & 0 &  \\ 
& \downarrow &  & \downarrow &  &  &  \\ 
& 0 &  & 0 &  &  & 
\end{array}%
$

Therefore: $P$ is isomorphic to the complex:

$...\rightarrow P_{n}\rightarrow P_{n-1}\rightarrow ...P_{2}\overset{d_{2}}{%
\rightarrow }P_{1}^{\prime }\oplus P_{0}^{\prime \prime }\oplus
P_{1}^{\prime \prime }\overset{d_{1}}{\rightarrow }P_{0}^{\prime }\oplus
P_{0}^{\prime \prime }\rightarrow 0$

with $\func{Im}d_{2}\subseteq Z_{1}^{\prime }\oplus P_{1}^{\prime \prime }.$

It follows $P$ decomposes as $P=P^{\prime }\oplus P^{\prime \prime }$ with:

$P\prime :..\rightarrow P_{n}\rightarrow P_{n-1}\rightarrow ...P_{2}\overset{%
d_{2}}{\rightarrow }P_{1}^{\prime }\oplus P_{1}^{\prime \prime }\overset{%
d_{1}}{\rightarrow }P_{0}^{\prime }\rightarrow 0$

$P^{\prime \prime }:$ $0\rightarrow P_{0}^{\prime \prime }\rightarrow
P_{0}^{\prime \prime }\rightarrow 0$

The projective $P_{0}^{\prime }$ is finitely generated.

Assume now $P=P^{\prime }\oplus P^{\prime \prime }$, where $H^{i}(P^{\prime
\prime })=0$ for all $i$ and

$P%
{\acute{}}%
:$ $.\rightarrow P_{n+1}\rightarrow P_{n}\rightarrow P_{n-1}\rightarrow
...P_{1}\rightarrow P_{0}\rightarrow 0$ with $P_{i}$ finitely generated for $%
0\leq i\leq n-2$.

Hence $B_{n-2}=\func{Im}d_{n-1}$ is finitely generated, therefore it has
finitely generated projective cover $P_{n-1}^{\prime }$ and as before, there
is a commutative exact diagram:

$%
\begin{array}{ccccccc}
& 0 &  & 0 &  & 0 &  \\ 
& \downarrow &  & \downarrow &  & \downarrow &  \\ 
0\rightarrow & Z_{n-1}^{\prime } & \rightarrow & P_{n-1}^{\prime } & 
\rightarrow & B_{n-2} & \rightarrow 0 \\ 
& \downarrow &  & \downarrow &  & \downarrow &  \\ 
0\rightarrow & Z_{n-1} & \rightarrow & P_{n-1} & \rightarrow & B_{n-2} & 
\rightarrow 0 \\ 
& \downarrow &  & \downarrow &  & \downarrow &  \\ 
0\rightarrow & P_{n-1}^{\prime \prime } & \rightarrow & P_{n-1}^{\prime
\prime } & \rightarrow & 0 &  \\ 
& \downarrow &  & \downarrow &  &  &  \\ 
& 0 &  & 0 &  &  & 
\end{array}%
$

Therefore: $Z_{n-1}\cong Z_{n-1}^{\prime }\oplus P_{n-1}^{\prime \prime }$.

Letting $B_{n-1}$be the image of $d_{n}$and $H_{n-1}$ the homology $%
H^{n-1}(P)$, which we assume finitely generated, there is an exact sequence: 
$0\rightarrow B_{n-1}\rightarrow Z_{n-1}^{\prime }\oplus P_{n-1}^{\prime
\prime }\rightarrow H_{n-1}\rightarrow 0$ and an induced commutative, exact
diagram:

$%
\begin{array}{ccccccc}
& 0 &  & 0 &  & 0 &  \\ 
& \downarrow &  & \downarrow &  & \downarrow &  \\ 
0\rightarrow & \overline{B}_{n-1} & \rightarrow & Z_{n-1}^{\prime } & 
\rightarrow & H_{n-1}^{\prime } & \rightarrow 0 \\ 
& \downarrow &  & \downarrow &  & \downarrow &  \\ 
0\rightarrow & B_{n-1} & \rightarrow & Z_{n-1}^{\prime }\oplus
P_{n-1}^{\prime \prime } & \rightarrow & H_{n-1} & \rightarrow 0 \\ 
& \downarrow &  & \downarrow &  & \downarrow &  \\ 
0\rightarrow & B_{n-1}^{\prime \prime } & \rightarrow & P_{n-1}^{\prime
\prime } & \rightarrow & H_{n-1}^{\prime \prime } & \rightarrow 0 \\ 
& \downarrow &  & \downarrow &  & \downarrow &  \\ 
& 0 &  & 0 &  & 0 & 
\end{array}%
$

with $\overline{B}_{n-1}=B_{n-1}\cap Z_{n-1}^{\prime }$ and $H_{n-1}^{\prime
\prime }$is finitely generated.

Therefore: the exact sequence: $0\rightarrow B_{n-1}^{\prime \prime }$:$%
\rightarrow $ $P_{n-1}^{\prime \prime }\rightarrow H_{n-1}^{\prime \prime
}\rightarrow 0$ is isomorphic to the direct sum of the exact sequences:

$0\rightarrow L_{n-1}$:$\rightarrow $ $Q_{n-1}^{\prime \prime }\rightarrow
H_{n-1}^{\prime \prime }\rightarrow 0$ and $0\rightarrow Q_{n-1}^{\prime
}\rightarrow Q_{n-1}^{\prime }\rightarrow 0\rightarrow 0$, with $%
Q_{n-1}^{\prime \prime }$ the projective cover of $H_{n-1}^{\prime \prime }$%
, hence finitely generated. Then $B_{n-1}^{\prime \prime }\cong
L_{n-1}\oplus Q_{n-1}^{\prime }$.

There is a commutative exact diagram:

$%
\begin{array}{ccccccc}
& 0 &  & 0 &  & 0 &  \\ 
& \downarrow &  & \downarrow &  & \downarrow &  \\ 
0\rightarrow & \overline{B}_{n-1} & \rightarrow & B_{n-1}^{\prime } & 
\rightarrow & L_{n-1} & \rightarrow 0 \\ 
& \downarrow &  & \downarrow &  & \downarrow &  \\ 
0\rightarrow & \overline{B}_{n-1} & \rightarrow & B_{n-1} & \rightarrow & 
L_{n-1}\oplus Q_{n-1}^{\prime } & \rightarrow 0 \\ 
& \downarrow &  & \downarrow &  & \downarrow &  \\ 
& 0 & \rightarrow & Q_{n-1}^{\prime } & \rightarrow & Q_{n-1}^{\prime } & 
\rightarrow 0 \\ 
&  &  & \downarrow &  & \downarrow &  \\ 
&  &  & 0 &  & 0 & 
\end{array}%
$

where $\overline{B}_{n-1}$and $L_{n-1}$ are finitely generated. It follows $%
B_{n-1}\cong $ $B_{n-1}^{\prime }\oplus Q_{n-1}^{\prime }$ with $%
B_{n-1}^{\prime }$finitely generated.

We have an exact sequence: $0\rightarrow B_{n-1}^{\prime }\oplus
Q_{n-1}^{\prime }\rightarrow P_{n-1}^{\prime }\oplus Q_{n-1}^{\prime }\oplus
Q_{n-1}^{\prime \prime }\rightarrow P_{n-2}.$

Taking the projective cover of $B_{n-1}^{\prime }$ we obtain an exact
sequence: $0\rightarrow Z_{n}^{\prime }\rightarrow P_{n}^{\prime
}\rightarrow B_{n-1}^{\prime }\rightarrow 0$. Therefore: $0\rightarrow
Z_{n}^{\prime }\rightarrow P_{n}^{\prime }\oplus Q_{n-1}^{\prime
}\rightarrow B_{n-1}^{\prime }\oplus Q_{n-1}^{\prime }\rightarrow 0$ is
exact.

As above, $P_{n}$ decomposes $P_{n}^{\prime }\oplus Q_{n-1}^{\prime }\oplus
P_{n}^{\prime \prime }.$

We have proved that $P$ decomposes in the direct sum of the complexes:

$.\rightarrow P_{n+1}\rightarrow P_{n}^{\prime }\oplus P_{n}^{\prime \prime
}\rightarrow P_{n-1}^{\prime }\oplus Q_{n-1}^{\prime \prime }\rightarrow
P_{n-2}...P_{1}\rightarrow P_{0}\rightarrow 0$

and $0\rightarrow Q_{n-1}^{\prime }\rightarrow Q_{n-1}^{\prime }\rightarrow
0....\rightarrow 0\rightarrow 0$, where $P_{n-1}^{\prime }\oplus
Q_{n-1}^{\prime \prime }$ is finitely generated.
\end{proof}

With the same hypothesis as in the previous lemma, let $X\in
D_{fg}^{b}(Gr_{A})$, we can choose a projective resolution of finitely
generated projective graded modules: $P\rightarrow X$ such that the
differential map $d_{j}:$ $P_{j}\rightarrow P_{j-1}$ has image contained in
the radical of $P_{j-1}$ .

Hence the complex $Hom_{A}(P$, $A_{0}):$

$0\rightarrow $ $Hom_{A}(P_{0}$, $A_{0})\rightarrow Hom_{A}(P_{1}$, $%
A_{0})\rightarrow ...Hom_{A}(P_{n}$, $A_{0})\rightarrow ...$ has zero
differential.

It follows $Ext_{A}^{k}(X,A_{0})=Hom_{A}(P_{k},A_{0})\neq 0$ and $%
Ext_{A}(X,A_{0})\neq 0$.

It follows $Ext$-$regular(X)\neq -\infty $, but $Ext$-$regular(X)=\infty $
is possible.

Assume $Ext$-$regular(X)=r$ is finite.

There is a left decomposition of $A$ in indecomposable summands: $A=\underset%
{i=1}{\overset{m}{\oplus }}Q_{i}$ and of each projective $P_{j}=\underset{i=1%
}{\overset{n}{\oplus }}Q_{i}^{(m_{i})}[-n_{i}^{j}]$ with $m_{i}\geq 0$ and $%
n_{i}^{j}$ integers.

Then $Ext_{A}^{j}(X,A_{0})=Hom_{A}(P_{j},A_{0})=\underset{i=1}{\overset{n}{%
\oplus }}D(Q_{i}/rQ_{i})^{(m_{i})}[n_{i}^{j}]$.

Therefore: $Hom_{A}(P_{j},A_{0})_{k}\neq 0$ if and only if for some $i$, $%
n_{i}^{j}+k=0$. Since the resolution is subdiagonal, $n_{i}^{j}\geq j$.

By definition $Ext_{A}^{j}(X,A_{0})_{\leq -r-1-j}=0$, this means $-r-j\leq
-n_{i}^{j}$ or $r\geq n_{i}^{j}-j$, for all $i$ and $r^{\prime }=\max
\{n_{i}^{j}-j\}$, exists.

Then $Ext_{A}^{j}(X,A_{0})_{\leq -r^{\prime }-j-1}=0$ and $%
Ext_{A}^{j}(X,A_{0})_{-(n_{j}^{i}-j)-j}\neq 0$.

We have proved $Ext$- $reg(X)=r=\max \{n_{i}^{j}-j\}$.

Let $P:...\rightarrow P_{n+1}\rightarrow P_{n}\rightarrow P_{n-1}\rightarrow
...P_{1}\rightarrow P_{0}\rightarrow A_{0}\rightarrow 0$ and $P^{\prime }:$ $%
...\rightarrow P_{n+1}^{\prime }\rightarrow P_{n}^{\prime }\rightarrow
P_{n-1}^{\prime }\rightarrow ...P_{1}^{\prime }\rightarrow P_{0}^{\prime
}\rightarrow A_{0}\rightarrow 0$ be minimal projective resolutions of $A_{0}$
as left and as right module, respectively.

Each $P_{j}$ has a decomposition $P_{j}=\underset{i=1}{\overset{m}{\oplus }}%
Q_{i}^{(m_{i})}[-n_{i}^{j}]$ and $Tor_{n}^{A}(A_{0},A_{0})$ is computed as
the $nth$-homology of the complex $A_{0}\otimes _{A}P:$

$...\rightarrow A_{0}\otimes _{A}P_{n+1}\rightarrow A_{0}\otimes
_{A}P_{n}\rightarrow A_{0}\otimes _{A}P_{n-1}\rightarrow ...A_{0}\otimes
_{A}P_{1}\rightarrow A_{0}\otimes _{A}P_{0}\rightarrow 0$ and $A_{0}\otimes
_{A}P_{n}=A_{0}\otimes _{A}\underset{i=1}{\overset{m}{\oplus }}%
Q_{i}^{(m_{i})}[-n_{i}^{n}]=$ $A/\mathfrak{m}\underset{i=1}{\otimes _{A}%
\overset{m}{\oplus }}Q_{i}^{(m_{i})}[-n_{i}^{n}]$ $\cong \underset{i=1}{%
\overset{m}{\oplus }}(Q_{i}/\mathfrak{m}Q_{i})^{(m_{i})}[-n_{i}^{n}]$ $\cong 
\underset{i=1}{\overset{m}{\oplus }}(S_{i})^{(m_{i})}[-n_{i}^{n}]$ and the
differential of $A_{0}\otimes _{A}P$ is zero.

Using the second resolution $Tor_{n}^{A}(A_{0},A_{0})$ is the $nth$-homology
of the complex $P^{\prime }\otimes _{A}A_{0}:$

$...\rightarrow P_{n+1}^{\prime }\otimes _{A}A_{0}\rightarrow P_{n}^{\prime
}\otimes _{A}A_{0}\rightarrow P_{n-1}^{\prime }\otimes _{A}A_{0}\rightarrow
...P_{1}^{\prime }\otimes _{A}A_{0}\rightarrow P_{0}^{\prime }\otimes
_{A}A_{0}\rightarrow 0$

Each $P_{j}^{\prime }$ has a decomposition $P_{j}^{\prime }=\underset{i=1}{%
\overset{m}{\oplus }}Q_{i}^{\prime (m_{i})}[-n_{i}^{\prime j}]$ and $%
P_{n}^{\prime }\otimes _{A}A_{0}=$\linebreak $\underset{i=1}{(\overset{m}{%
\oplus }}Q_{i}^{\prime (m_{i})}[-n_{i}^{\prime j}])\otimes _{A}A_{0}=%
\underset{i=1}{\overset{m}{\oplus }}(Q_{i}^{\prime }/(Q_{i}^{\prime })%
\mathfrak{m}^{(m_{i})}[-n_{i}^{\prime j}]$ $\cong \underset{i=1}{\overset{m}{%
\oplus }}(S_{i}^{\prime })^{(m_{i})}[-n_{i}^{\prime j}]$ and the
differential of $P^{\prime }\otimes _{A}A_{0}$ is zero.

It follows $n_{i}^{j}=n_{i}^{\prime j}$ for all $i$.

By the above remark, $Ext$- $reg_{A}A_{0}=Ext$- $regA_{0A}=Ext$- $regA_{0}.$

We write this as a theorem.

\begin{theorem}
Let $A$ be a locally finite $\Bbbk $-algebra. Then $Ext$- $reg_{A}A_{0}=Ext$%
- $regA_{0A}=Ext$- $regA_{0}$.
\end{theorem}

We next have:

\begin{theorem}
Let $A$ be a noetherian graded AS Gorenstein algebra of finite local
cohomology dimension. Given $X\in D_{fg}^{b}(Gr_{A})$, $X\neq 0$. Then $Ext$%
- $reg(X)\leq CMreg(X)+Ext$- $regA_{0}$.
\end{theorem}

\begin{proof}
We proved above $CMreg(X)\neq -\infty $. If $Ext$- $regA_{0}=\infty $, then
the inequality is trivially satisfied.

We may assume $Ext$- $regA_{0}=r$ is finite .

Let $P\rightarrow A_{0}$ be a minimal projective resolution. Changing
notation,

$P:$ $...P^{(n+1)}\rightarrow $ $P^{(n)}\rightarrow ...P^{(1)}\rightarrow
P^{(0)}\rightarrow 0$

where $P^{(m)}=\oplus P_{j}^{(m)}[-\sigma _{m,j}]$ and $\sigma _{m,j}\leq
r+m.$

Dualizing, we obtain an injective resolution $I$ with $I^{m}=\oplus
D(P_{j}^{(m)})[\sigma _{m,j}]$, of $A_{0}$ as right module.

Let $p$ be $p=CMreg(X)$, $Z=R\Gamma _{\mathfrak{m}}(X)$ and denote by $%
h^{-n} $ the homology. Then by definition we have:

$h^{-n}(Z)_{\geq p+1+n}=h^{-n}(R\Gamma _{\mathfrak{m}}(X))_{\geq
p+1+n}=\Gamma _{\mathfrak{m}}^{-n}(X)_{\geq p+1+n}=0$ for all $n$.
Therefore: ($h^{-n}(Z))_{\leq -p-1-n}^{\prime }=0$.

But $Ext_{A}^{m}(h^{-n}(Z)$, $A_{0})$ is a subquotient of $Hom_{A}($ $%
h^{-n}(Z)$, $I^{m})=Hom_{A}($ $h^{-n}(Z)$, $\oplus D(P_{j}^{(m)})[\sigma
_{m,j}]$)=$\oplus Hom_{A}($ $h^{-n}(Z)$, $D(P_{j}^{(m)})[\sigma _{m,j}]$)$%
\cong $

$\oplus Hom_{\Bbbk }($ $(P_{j}^{(m)})^{\ast }\otimes h^{-n}(Z),\Bbbk
)[\sigma _{m,j}]\cong \oplus Hom_{\Bbbk }($ $(e_{j}h^{-n}(Z),\Bbbk )[\sigma
_{m,j}]$ with $e_{j}$ the idempotent corresponding to $P_{j}^{(m)}$.

Since ($h^{-n}(Z))_{\leq -p-1-n}^{\prime }=0$, it follows $Hom_{\Bbbk }($ $%
(e_{j}h^{-n}(Z),\Bbbk )_{\leq -p-1-n}=0$.

Observe that the truncation of a shifted module $M[k]_{\leq -t-k}=M_{\leq
-t}[k]$ .

Therefore: $Ext_{A}^{m}(h^{-n}(Z)$, $A_{0})_{\leq }-p-1-n-r-m=0.$

We have a converging spectral sequence:

$E_{2}^{m,n}=Ext_{A}^{m}(h^{-n}(Z)$, $A_{0})\Longrightarrow Ext_{A}^{m+n}(Z$%
, $A_{0})$.

This means $Ext_{A}^{m+n}(Z$, $A_{0})$ is a subquotient of $%
E_{2}^{m,n}=Ext_{A}^{m}(h^{-n}(Z)$, $A_{0})$ and $Ext_{A}^{m}(h^{-n}(Z)$, $%
A_{0})_{\leq }-p-1-r-(n+m)=0$ implies $Ext_{A}^{q}(Z$, $A_{0})_{\leq
-p-1-r-q}=0$ for all $q$.

We have isomorphisms: $Ext_{A}^{q}(Z$, $A_{0})=Ext_{A}^{q}(R\Gamma _{%
\mathfrak{m}}(X)$, $A_{0})=$\linebreak $H^{q}(RHom(R\Gamma _{\mathfrak{m}%
}(X),A_{0}))\cong H^{q}(RHom(X,A_{0}))=Ext_{A}^{q}(X,A_{0})$.

Therefore: $Ext_{A}^{q}(X,A_{0})_{\leq -p-1-r-q}=0$.

This implies $Ext$-$reg(X)\leq p+r=CMreg(X)+Ext$- $regA_{0}$.
\end{proof}

\begin{corollary}
Assume the same conditions as in the theorem and $Ext$- $regA_{0}$ finite.
Then for any $X\in D_{fg}^{b}(Gr_{A})$, $Ext$- $reg(X)$ is finite.
\end{corollary}

\begin{proof}
This follows from the above remark that $CMreg(X)$ is finite.
\end{proof}

Interchanging the roles of Ext-regular and CM-regular we obtain in the next
result a similar inequality.

\begin{theorem}
Let $A$ be a noetherian AS Gorenstein algebra of finite local cohomology
dimension. Given $X\in D_{fg}^{b}(Gr_{A})$, $X\neq 0$. Then $CMreg(X)\leq
Ext $-$reg(X)+CMregA$.
\end{theorem}

\begin{proof}
Since we know $CMregA\neq -\infty $, the assumption $Ext$-$reg(X)=\infty $
gives the inequality and we can assume $Ext$-$reg(X)=r$ is finite.

As before, there is a projective resolution $P\rightarrow X$ of $X$ with $%
P^{(m)}=\oplus P_{j}^{(m)}[-\sigma _{m,j}]$ and $\sigma _{m,j}\leq r+m.$

Let $p$ be $p=CMreg_{A}A=CMregA_{A}$. Then by definition $\Gamma _{\mathfrak{%
m}^{op}}^{n}(A)_{\geq p+1-n}=0$ for all $n$.

Tor$_{-m}^{A}(\Gamma _{\mathfrak{m}^{op}}^{n}(A)$,$X)$ is a subquotient of $%
\Gamma _{\mathfrak{m}^{op}}^{n}(A)\otimes _{A}P^{(-m)}=\oplus \Gamma _{%
\mathfrak{m}^{op}}^{n}(A)\otimes _{A}P_{j}^{(-m)}[-\sigma _{-m,j}]=\oplus
\Gamma _{\mathfrak{m}^{op}}^{n}(A)e_{j}[-\sigma _{-m,j}]$ with $e_{j}$ the
idempotent corresponding to $P_{j}^{(-m)}$ and $\sigma _{-m,j}\leq r-m.$

Therefore: $\Gamma _{\mathfrak{m}^{op}}^{n}(A)[-\sigma _{-m,j}]_{\geq
p+1-n+(r-m)}=0.$

As above, it follows Tor$_{-m}^{A}(\Gamma _{\mathfrak{m}^{op}}^{n}(A)$,$%
X)_{\geq p+1-n+r-m}=0$

The spectral sequence $E_{2}^{-m.n}=$Tor$_{-m}^{A}(\Gamma _{\mathfrak{m}%
^{op}}^{n}(A)$,$X)\Longrightarrow \Gamma _{\mathfrak{m}}^{-m+n}(X)$
converges (Lemma 3).

Hence $\Gamma _{\mathfrak{m}}^{m+n}(X)$ is a subquotient of Tor$%
_{-m}^{A}(\Gamma _{\mathfrak{m}^{op}}^{n}(A)$,$X)$ and it follows\linebreak\ 
$\Gamma _{\mathfrak{m}}^{q}(X)_{\geq p+1+r-q}=0.$

We have proved $CMreg(X)\leq p+r=Ext$-$reg(X)+CMregA$.
\end{proof}

\begin{remark}
The algebra $A$ is Koszul if and only if $Ext$-$regA_{0}=0$.
\end{remark}

\begin{corollary}
Assume the same conditions on $A$ as in the theorem and in addition $A$
Koszul and $CMregA=0$. Then $Ext$- $reg(X)=CMreg(X)$.
\end{corollary}

We have all the ingredients to prove the main theorem of the section.

\begin{theorem}
Let $A$ be a noetherian AS Gorenstein algebra of finite local cohomology
dimension. Assume $A$ Koszul and let $M$ be a finitely generated graded $A$%
-module . Then for $s\geq CMregM$, the projective resolution of $M_{\geq
s}[s]$ is linear.
\end{theorem}

\begin{proof}
Assume $M_{\geq s}[s]\neq 0$ and let $P^{(n+1)}\rightarrow $ $%
P^{(n)}\rightarrow ...P^{(1)}\rightarrow P^{(0)}\rightarrow M_{\geq
s}[s]\rightarrow 0$ be the projective resolution. The module $M_{\geq s}[s]$
is generated in degree zero and $P^{(m)}$ decomposes as $P^{(m)}=\oplus
P_{j}^{(m)}[-\sigma _{m,j}\ ]$ and $m\leq \sigma _{m,j}$.

We most prove $P^{(m)}$ does not have generators in degrees larger than $m$,
or equivalently $Ext$- $reg(M_{\geq s}[s])\leq 0$, which will follow from
the above inequalities once we prove $CMreg(M_{\geq s}[s])\leq 0$ or
equivalently, $CMreg(M_{\geq s})\leq s$, this is:

$\Gamma _{\mathfrak{m}}^{m}(M_{\geq s})_{\geq s+1-m}=0$.

The module $L=M/M_{\geq s}$ is of finite length. By the local cohomology
formula, $\underset{k}{\underrightarrow{\lim }}Ext_{A}^{j}(A/\mathfrak{m}%
^{k} $,$L)=D(Ext_{A}^{n-j}(L,D(\Gamma _{\mathfrak{m}^{{}}}^{n}(A)))$.

Since $A$ is graded AS Gorenstein $Ext_{A}^{n-j}(L,D(\Gamma _{\mathfrak{m}%
^{{}}}^{n}(A))=0$ for $j\neq n$. It follows $\Gamma _{\mathfrak{m}%
}^{j}(M/M_{\geq s})=\left\{ 
\begin{array}{ccc}
0 & \text{if} & j\neq s \\ 
M/M_{\geq s} & \text{if} & j=s%
\end{array}%
\right. $

The exact sequence: $0\rightarrow M_{\geq s}\rightarrow M\rightarrow
M/M_{\geq s}\rightarrow 0$ induces a triangle $M_{\geq s}\rightarrow
M\rightarrow M/M_{\geq s}\rightarrow M_{\geq s}[1]$ , hence a triangle $%
R\Gamma _{\mathfrak{m}}(M_{\geq s})\rightarrow R\Gamma _{\mathfrak{m}%
}(M)\rightarrow R\Gamma _{\mathfrak{m}}(M/M_{\geq s})\rightarrow R\Gamma _{%
\mathfrak{m}}(M_{\geq s})[1]$, by the long homology sequence we obtain an
exact sequence:

$\rightarrow \Gamma _{\mathfrak{m}}^{m-1}(M/M_{\geq s})\rightarrow \Gamma _{%
\mathfrak{m}}^{m}(M_{\geq s})\rightarrow \Gamma _{\mathfrak{m}%
}^{m}(M)\rightarrow \Gamma _{\mathfrak{m}}^{m}(M/M_{\geq s})$

The inequality $s\geq $ $CMreg(M)$ implies $\Gamma _{\mathfrak{m}%
}^{m}(M)_{\geq s+1-m}=0$ for all $m$. Since $M/M_{\geq s}$ has length $s$, $%
\Gamma _{\mathfrak{m}}^{m}(M/M_{\geq s})_{\geq s+1-m}=0$ for all $m.$

It follows $\Gamma _{\mathfrak{m}}^{m}(M_{\geq s})_{\geq s+1-m}=0$ for all $%
m $.
\end{proof}

\section{Algebras AS Gorenstein and Koszul}

In this section we will use the main theorem of the last section in order to
extend a theorem by Bernstein-Gelfand-Gelfand, [3] which claims that for the
exterior algebra in $n$-variables $\Lambda $ there is an equivalence of
triangulated categories \underline{$gr$}$_{\Lambda }$ $\cong D^{b}(CohP_{n})$
from the stable category of finitely generated graded modules to the
category of bounded complexes of coherent sheaves on projective space $P_{n}$%
. The theorem was extended to finite dimensional Koszul algebras in $[15],$%
[16] see also [21]. We want to prove here a version of this theorem for AS
Gorenstein algebras of finite cohomological dimension. We will show that the
arguments used in [15] can be easily extended to this situation. We will
assume the reader is familiar with the results of [13], [15] and [17] and
the bibliography given there.

It was proved in [25] and [12] that a finite dimensional Koszul algebra $%
\Lambda $ is selfinjective if and only if its Yoneda algebra $\Gamma $ is
Artin Schelter regular [1]. The following generalization was proved in [13]
and [22] :

\begin{theorem}
A Koszul algebra $\Lambda $ is graded AS\ Gorenstein if and only if its
Yoneda algebra $\Gamma $ is graded AS Gorenstein.
\end{theorem}

\begin{remark}
Observe the following:

i) The algebra $\Lambda $ can be noetherian with non noetherian Yoneda
algebra .

ii) The algebra $\Lambda $ could be Gorenstein and $\Gamma $ only weakly
Gorentein this is: there exists an integer n such that for all $\Gamma $%
-modules left (right) of finite length \linebreak $Ext_{\Gamma }^{j}(M$ ,$%
\Gamma )=0$ for all $j>n$.

iii) The algebra $\Lambda $ could be of finite local cohomology dimension
and $\Gamma $ of infinite local cohomology dimension.
\end{remark}

However, there are Koszul algebras $\Lambda $ with Yoneda algebra $\Gamma $
such that both $\Lambda $ and $\Gamma $ are graded AS\ Gorenstein,
noetherian (in both sides) and of finite cohomological dimension, for
example if $\Lambda $ is selfinjective with noetherian Yoneda algebra $%
\Gamma $ then $\Lambda \otimes \Gamma $ is AS Gorenstein Koszul noetherian
of finite local cohomology dimension on both sides with Yoneda algebra the
skew tensor product (in the sense of [5] or [18]) $\Lambda \boxtimes \Gamma $%
which is also AS Gorenstein noetherian and of finite local cohomology
dimension on both sides.

A concrete example of such algebras is $\Lambda $ the exterior algebra in $n$
variables and $\Gamma $ the polynomial algebra in $n$ variables, this
example appears as the cohomology ring of an elementary abelian $p$-group
over a field of positive characteristic $p\neq 2$. [4]

Another example is the trivial extension $\Lambda =\Bbbk Q\rhd D(\Bbbk Q)$
with $Q$ an Euclidean diagram and $\Gamma $ the preprojective algebra
corresponding to $Q$ [11].

\bigskip We need the following definitions and results from [17]:

\begin{definition}
Let $\Lambda $ be a Koszul algebra with graded Jacobson radical $\mathfrak{m}
$. A finitely generated graded $\Lambda $-module $M$ is weakly Koszul if it
has a minimal projective resolution:

$\rightarrow P_{n}\overset{d_{n}}{\rightarrow }P_{n-1}\rightarrow
...P_{1}\rightarrow P_{0}\overset{d_{0}}{\rightarrow }M\rightarrow 0$ such
that $\mathfrak{m}^{k+1}P_{i}\cap \ker d_{i}=\mathfrak{m}^{k}\ker d_{i}$.
\end{definition}

The next result characterizing weakly Koszul modules was proved in [17].

\begin{theorem}
Let $\Lambda $ be a Koszul algebra with Yoneda algebra and denote by $%
gr_{\Lambda }$, the category of finitely generated graded $\Lambda $%
-modules, $F:gr_{\Lambda }\rightarrow Gr_{\Gamma }$ be the exact functor $%
F(M)=\underset{k\geq 0}{\oplus }Ext_{\Lambda }^{k}(M,\Lambda _{0})$. Then $M$
is weakly Koszul if and only if $F(M)$ is Koszul .
\end{theorem}

As a consequence of this theorem and the results of the last section we have:

\begin{theorem}
Let $\Lambda $ be a Koszul algebra with Yoneda algebra $\Gamma $ such that
both are AS graded Gorenstein noetherian algebras of finite local cohomology
dimension on both sides. Then given a finitely generated left $\Lambda $%
-module $M$ there is a non negative integer $k$ such that $\Omega ^{k}(M)$
is weakly Koszul.
\end{theorem}

\begin{proof}
Since $\Lambda $ is Koszul AS graded Gorenstein noetherian algebras of
finite local cohomology dimension on both sides, for any finitely generated
graded $\Lambda $-module $M$ there is a truncation $M_{\geq s}$ such that $%
M_{\geq s}[s]$ is Koszul and there is an exact sequence: $0\rightarrow
M_{\geq s}\rightarrow M\rightarrow M/M_{\geq s}\rightarrow 0$ with $%
M/M_{\geq s}$ of finite length. Then we have an exact sequence: $F(M/M_{\geq
s})\rightarrow F(M)\rightarrow F(M_{\geq s})$. Since $F$ sends simple
modules to indecomposable projective, it sends modules of finite length to
finitely generated modules and $M_{\geq s}$ Koszul up to shift implies $%
F(M_{\geq s})$ Koszul up to shift, hence finitely generated. Since we are
assuming $\Gamma $ noetherian, it follows $F(M)$ is finitely generated. By
Theorem 6, $F(M)$ has a truncation $F(M)_{\geq t}$ Koszul up to shift and $%
F(M)_{\geq t}=\underset{k\geq t}{\oplus }Ext_{\Lambda }^{k}(M,\Lambda
_{0})[-t]\cong \underset{k\geq 0}{\oplus }Ext_{\Lambda }^{k}(\Omega
^{t}(M),\Lambda _{0})[-t]=F(\Omega ^{t}(M)).$

By Theorem 8, $\Omega ^{t}(M)$ is weakly Koszul.
\end{proof}

\begin{definition}
A complex of graded $\Lambda $-modules is linear if for each i, the ith
module is generated in degree i , provided is not zero.
\end{definition}

Let $Q$ be a finite quiver, $\Bbbk Q$ the path algebra graded by path length
and $\Lambda =\Bbbk Q/I$ be a quotient with $I$ a homogeneous ideal
contained in $\Bbbk Q_{\geq 2}$ and $\Gamma $ the Yoneda algebra of $\Lambda 
$, it was shown in [16] that there is a functor

\begin{center}
$%
\begin{array}{c}
\Phi :\ell .f.gr_{\Lambda }\rightarrow \mathfrak{lcp}_{\Gamma }^{-}%
\end{array}%
$
\end{center}

between the category of locally finite graded $\Lambda $-modules, $\ell
.f.gr_{\Lambda }$, and the category of right bounded linear complexes of
finitely generated graded projective $\Gamma $- modules $\mathfrak{lcp}%
_{\Gamma }^{-}$. We recall the construction of $\Phi .$

Let $M=\{M_{i}\}_{i\geq n_{0}}$ be a finitely generated graded $\Lambda $%
-module and $\mu :\Lambda _{1}\otimes _{\Lambda _{0}}M_{k}\rightarrow
M_{k+1} $ the map of $\Lambda _{0}$-modules given by multiplication.

Since $M_{k}$ is a finitely generated $\Lambda _{0}$-module, we have a
homomorphism of $\Lambda _{0}$-modules

\begin{center}
$%
\begin{array}{c}
D(\mu ):D(M_{k+1})\rightarrow D(M_{k})\otimes _{\Lambda _{0}}D(\Lambda _{1})%
\end{array}%
$,
\end{center}

where $D(-)=Hom_{\Lambda _{0}}(-,\Lambda _{0})$. Applying $Hom_{\Lambda
}(-,\Lambda _{0})$ to the exact sequence

\begin{center}
$%
\begin{array}{c}
0\rightarrow \mathfrak{m}\rightarrow \Lambda \rightarrow \Lambda
_{0}\rightarrow 0%
\end{array}%
$
\end{center}

induces an exact sequence

\begin{center}
$%
\begin{array}{c}
\text{0}\rightarrow \text{Hom}_{\Lambda }\text{(}\Lambda _{0}\text{,}\Lambda
_{0}\text{)}\rightarrow \text{Hom}_{\Lambda }\text{(}\Lambda \text{,}\Lambda
_{0}\text{)}\rightarrow \text{Hom}_{\Lambda }\text{(}\mathfrak{m}\text{,}%
\Lambda _{0}\text{)}\rightarrow \text{Ext}_{\Lambda }^{1}\text{(}\Lambda _{0}%
\text{,}\Lambda _{0}\text{)}\rightarrow \text{0}%
\end{array}%
$
\end{center}

the second map is an isomorphism, which implies Hom$_{\Lambda }$($\mathfrak{m%
}$,$\Lambda _{0}$)$\rightarrow $Ext$_{\Lambda }^{1}$($\Lambda _{0}$,$\Lambda
_{0}$) is an isomorphism. Since $\Lambda _{0}$ is semisimple, there is an
isomorphism

\begin{center}
$%
\begin{array}{c}
Hom_{\Lambda }(\mathfrak{m},\Lambda _{0})\cong Hom_{\Lambda }(\mathfrak{m/m}%
^{2},\Lambda _{0})%
\end{array}%
$
\end{center}

As a result there is an isomorphism $D(\Lambda _{1})=$Hom$_{\Lambda
_{0}}(\Lambda _{1},\Lambda _{0})\cong \Gamma _{1}$ and we have a $\Lambda
_{0}$-linear map $d_{k_{0}}:D(M_{k+1})\rightarrow D(M_{k})\otimes _{\Lambda
_{0}}\Gamma _{1}.$

For any $\ell \geq 0$, using the fact $\Lambda _{0}\cong \Gamma _{0}$ the
multiplication map $\upsilon :\Gamma _{1}\otimes _{\Gamma _{0}}\Gamma _{\ell
}\rightarrow \Gamma _{\ell +1}$ induces a new map $d_{k_{\ell }}$, as shown
in the diagram:

\begin{center}
$%
\begin{array}{ccc}
D(M_{k+1})\otimes _{\Gamma _{0}}\Gamma _{\ell } & \rightarrow & 
D(M_{k})\otimes _{\Gamma _{0}}\Gamma _{1}\otimes _{\Gamma _{0}}\Gamma _{\ell
} \\ 
& \searrow & \downarrow 1\otimes \upsilon \\ 
d_{k_{\ell }} &  & D(M_{k})\otimes _{\Gamma _{0}}\Gamma _{\ell +1}%
\end{array}%
$
\end{center}

Hence there is a map in degree zero

\begin{center}
$%
\begin{array}{c}
d_{k}:D(M_{k+1})\otimes _{\Gamma _{0}}\Gamma \text{[-k-1]}\rightarrow
D(M_{k})\otimes _{\Gamma _{0}}\Gamma \text{[-k]}%
\end{array}%
$
\end{center}

\begin{definition}
We call $\Phi $ the linearization functor.
\end{definition}

\begin{proposition}
The sequence $\Phi (M)=\{D(M_{k+1})\otimes _{\Gamma _{0}}\Gamma $[-k-1], $%
d_{k}\}$ is a right bounded linear complex of finitely generated graded
projective $\Gamma $-modules.
\end{proposition}

The following proposition was proved in [16]

\begin{proposition}
The algebra $\Lambda =\Bbbk Q/I$ is quadratic if and only if\linebreak\ $%
\Phi :\ell .f.gr_{\Lambda }\rightarrow \mathfrak{lcp}_{\Gamma }^{-}$ is a
duality.
\end{proposition}

We can say more in case $\Lambda =\Bbbk Q/I$ is a Koszul algebra.

\begin{theorem}
Suppose $\Lambda =\Bbbk Q/I$ is a Koszul algebra and $M$ a locally finite
bounded above graded $\Lambda $-module. Then $M$ is Koszul if and only if $%
\Phi (M)$ is exact, except at minimal degree; in that case, $\Phi (M)$ is a
minimal projective resolution of the Koszul module (up to shift) $F(M)=%
\underset{k\geq t}{\oplus }Ext_{\Lambda }^{k}(M,\Lambda _{0})$.
\end{theorem}

\subsection{Approximations by linear complexes}

In this section we will see that the approximations by linear complexes
given in [15] can be extended to the family of AS Gorenstein Koszul algebras
considered above. Let $\Lambda $ be a possibly infinite dimensional Koszul
algebra with Yoneda algebra $\Gamma $. The category of complexes of finitely
generated graded projective $\Gamma $-modules with bounded homology $%
K^{-b}(grP_{\Gamma }),$ module the homotopy relations, is equivalent to the
derived category of bounded complexes $D_{fg}^{b}(Gr_{\Gamma })$.

We proved in Lemma 4, that any complex $X$ in $D_{fg}^{-}(Gr_{\Gamma })$ has
projective resolution $P\rightarrow X$ with $P$ subdiagonal. Linear
complexes are by definition subdiagonal.

\begin{lemma}
Let $M$ and $N$ be complexes of graded modules over a graded algebra and $%
f:M\rightarrow N$ a null-homotopic chain map. If $M$ is linear and $N$ is
diagonal, then $f=0$.
\end{lemma}

\begin{corollary}
Any morphism in a derived category of modules whose domain is a bounded on
the right linear complex of projective modules can be represented by a chain
map.
\end{corollary}

Since our interest is in Koszul algebras we need the following:

\begin{definition}
A complex is said to be totally linear, if it is linear and each of its
terms has a linear projective resolution.

Observe that this notion is a generalization of a linear complex of
projective modules.
\end{definition}

Observe that, though the proposition below has been stated more generally
than in [15], the proof is the same as in [15].

\begin{proposition}
Let $\Gamma $be a noetherian graded ring and $M_{\bullet }=\{M_{i}$, $%
d_{i}\}_{n\geq i\geq 0}$ a bounded totally linear complex of finitely
generated graded $\Gamma $-modules. Then there exists a bounded on the right
linear complex of finitely generated projective graded modules $P_{\bullet }$
and a quasi-isomorphism $\mu :P_{\bullet }\rightarrow M_{\bullet }$ such
that $\mu _{i}:P_{i}\rightarrow M_{i}$ is an epimorphism for each i.
\end{proposition}

\begin{proof}
The approximation is constructed by induction. We start with the exact
sequence: $0\rightarrow B_{0}\rightarrow M_{0}\rightarrow H_{0}\rightarrow 0$%
, take the projective cover $P_{0}\rightarrow M_{0}\rightarrow 0$ and
complete a commutative exact diagram:

$%
\begin{array}{ccccccc}
& 0 &  & 0 &  &  &  \\ 
& \downarrow &  & \downarrow &  &  &  \\ 
0\rightarrow & \Omega (M_{0}) & \rightarrow & \Omega (M_{0}) & \rightarrow & 
0 &  \\ 
& \downarrow &  & \downarrow &  &  &  \\ 
0\rightarrow & \Omega (H_{0}) & \rightarrow & P_{0} & \rightarrow & H_{0} & 
\rightarrow 0 \\ 
& \downarrow &  & \downarrow &  & \downarrow 1 &  \\ 
0\rightarrow & B_{0} & \rightarrow & M_{0} & \rightarrow & H_{0} & 
\rightarrow 0 \\ 
& \downarrow &  & \downarrow &  &  &  \\ 
& 0 &  & 0 &  &  & 
\end{array}%
$

Taking the pull back we obtain a commutative exact diagram:

$%
\begin{array}{ccccccc}
&  &  & 0 &  & 0 &  \\ 
&  &  & \downarrow &  &  &  \\ 
& 0 & \rightarrow & \Omega (M_{0}) & \rightarrow & \Omega (M_{0}) & 
\rightarrow 0 \\ 
& \downarrow &  & \downarrow &  & \downarrow &  \\ 
0\rightarrow & Z_{1} & \rightarrow & W_{1} & \rightarrow & \Omega (H_{0}) & 
\rightarrow 0 \\ 
& \downarrow &  & \downarrow &  & \downarrow &  \\ 
0\rightarrow & Z_{1} & \rightarrow & M_{1} & \rightarrow & B_{0} & 
\rightarrow 0 \\ 
& \downarrow &  & \downarrow &  & \downarrow &  \\ 
& 0 &  & 0 &  & 0 & 
\end{array}%
$

Since $M_{1}$ and $\Omega (M_{0})$ are both generated in degree one and have
linear resolutions, the same is true for $W_{1}.$

It is clear that the complex $0\rightarrow M_{n}\rightarrow ...\rightarrow
M_{2}\rightarrow W_{1}\rightarrow P_{0}\rightarrow 0$ is totally linear and
quasi-isomorphic to $M_{\bullet }$and the quasi-isomorphism is an
epimorphism in each degree.

Assume by induction we have constructed the totally linear complex: $%
0\rightarrow M_{n}\rightarrow ...\rightarrow M_{j+1}\rightarrow
W_{j}\rightarrow P_{j-1}\rightarrow ...\rightarrow P_{0}\rightarrow 0$

together with a quasi-isomorphism $\mu $ to the complex $M_{\bullet }$ which
is an epimorphism in each degrees $k$ with $0\leq k\leq j$ and the identity
in degrees $k$ for $j+1\leq k\leq n$.

We have a commutative exact diagram:

$%
\begin{array}{ccccccc}
& 0 &  & 0 &  &  &  \\ 
& \downarrow &  & \downarrow &  &  &  \\ 
0\rightarrow & \Omega (W_{j}) & \rightarrow & \Omega (W_{j}) & \rightarrow & 
0 &  \\ 
& \downarrow &  & \downarrow &  & \downarrow &  \\ 
0\rightarrow & K & \rightarrow & P_{j} & \rightarrow & W_{j}/B_{j} & 
\rightarrow 0 \\ 
& \downarrow &  & \downarrow &  & \downarrow 1 &  \\ 
0\rightarrow & B_{j} & \rightarrow & W_{j} & \rightarrow & W_{j}/B_{j} & 
\rightarrow 0 \\ 
& \downarrow &  & \downarrow &  & \downarrow &  \\ 
& 0 &  & 0 &  & 0 & 
\end{array}%
$

which induces by pullback the commutative exact diagram:

$%
\begin{array}{ccccccc}
&  &  & 0 &  & 0 &  \\ 
&  &  & \downarrow &  & \downarrow &  \\ 
& 0 & \rightarrow & \Omega (W_{j}) & \rightarrow & \Omega (W_{j}) & 
\rightarrow 0 \\ 
& \downarrow &  & \downarrow &  & \downarrow &  \\ 
0\rightarrow & Z_{j+1} & \rightarrow & W_{j+1} & \rightarrow & K & 
\rightarrow 0 \\ 
& \downarrow &  & \downarrow &  & \downarrow &  \\ 
0\rightarrow & Z_{j+1} & \rightarrow & M_{j+1} & \rightarrow & B_{j} & 
\rightarrow 0 \\ 
& \downarrow &  & \downarrow &  & \downarrow &  \\ 
& 0 &  & 0 &  & 0 & 
\end{array}%
$

By Verdier's lemma we have a complex: $P_{\bullet }^{(j)}$ : $0\rightarrow
M_{n}\rightarrow ...\rightarrow M_{j+2}\rightarrow W_{j+1}\rightarrow
P_{j}\rightarrow ...\rightarrow P_{0}\rightarrow 0$ and a quasi isomorphism $%
\dot{\mu}:P_{\bullet }^{(j)}\rightarrow M_{\bullet }$ which is the identity
in degrees $k$ such that $j+2\leq k\leq n$ and an epimorphism in the
remaining degrees.

We get by induction a totally linear complex: $P_{\bullet }^{(n-1)}$ : $%
0\rightarrow W_{n}\rightarrow P_{n-1}\rightarrow P_{n-2}\rightarrow
...\rightarrow P_{0}\rightarrow 0$ with $P_{j}$ for $0\leq j\leq n-1$
finitely generated graded projective modules generated in degree $j$. There
is a quasi-isomorphism $\mu :P_{\bullet }^{(n-1)}\rightarrow M_{\bullet }$
such that in each degree the maps are epimorphisms.

As above, we obtain the commutative exact diagram:

$%
\begin{array}{ccccccc}
& 0 &  & 0 &  &  &  \\ 
& \downarrow &  & \downarrow &  &  &  \\ 
0\rightarrow & \Omega (W_{n}) & \rightarrow & \Omega (W_{n}) & \rightarrow & 
0 &  \\ 
& \downarrow &  & \downarrow &  & \downarrow &  \\ 
0\rightarrow & Z_{n}^{\prime } & \rightarrow & P_{n} & \rightarrow & B_{n-1}
& \rightarrow 0 \\ 
& \downarrow &  & \downarrow &  & \downarrow &  \\ 
0\rightarrow & Z_{n} & \rightarrow & W_{n} & \rightarrow & B_{n-1} & 
\rightarrow 0 \\ 
& \downarrow &  & \downarrow &  & \downarrow &  \\ 
& 0 &  & 0 &  & 0 & 
\end{array}%
$

Since $W_{n}$ has a linear resolution $\Omega (W_{n})$ has a linear
resolution $P_{\bullet }^{(n+1)}\rightarrow \Omega (W_{n}).$

It follows $P_{\bullet }^{(n+1)}\rightarrow P_{n}\rightarrow
P_{n-1}\rightarrow P_{n-2}\rightarrow ...\rightarrow P_{0}\rightarrow 0$ is
a linear complex of finitely generated graded projective modules which is
quasi-isomorphic to $M_{\bullet }$ and all the maps in the quasi-isomorphism
are epimorphisms.
\end{proof}

We see next that for noetherian AS Gorenstein algebras of finite local
cohomology any bounded complex can be approximated by a totally linear
complex.

\begin{proposition}
Let $\Gamma $ be a Koszul algebra AS graded Gorenstein noetherian algebras
of finite local cohomology dimension on both sides. Then given a bounded
complex $M_{\bullet }$ of finitely generated graded $\Gamma $-modules, there
exists a totally linear subcomplex $L_{\bullet }$ such that $M_{\bullet }/$ $%
L_{\bullet }$ is a complex of modules of finite length.
\end{proposition}

\begin{proof}
Let $M_{\bullet }$ be the complex $M_{\bullet }=\{M_{j}\mid 0\leq j\leq n\}.$
By Theorem 6, for each $j$ there is a truncation ($M_{j})_{\geq n_{j}}$ such
that ($M_{j})_{\geq n_{j}}[n_{j}]$ is Koszul. Taking $n=\{\max $ $n_{j}\}$
each ($M_{j})_{\geq n}[n]$ is Koszul. Define $L_{\bullet }=\{L_{j}\mid
L_{j}= $($M_{j})_{\geq n+j}\}.$ Then $L_{\bullet }$ is totally linear with $%
M_{\bullet }/L_{\bullet }$ a is a complex of modules of finite length.
\end{proof}

We have now the following:

\begin{lemma}
Let $\Lambda $ be a Koszul algebra AS graded Gorenstein noetherian algebras
of finite local cohomology dimension on both sides with Yoneda algebra $%
\Gamma $and $\Phi :gr_{\Lambda }\rightarrow \mathfrak{lcp}_{\Gamma }^{-}$
the linearization functor. Then for any finitely generated module $M$ the
complex $\Phi (M)$ is contained in $\mathfrak{lcp}_{\Gamma }^{-,b}$, this is
the homology $H^{i}(\Phi (M))=0$ for almost all i.
\end{lemma}

\begin{proof}
According to Theorem 6, there is a truncation $M_{\geq s}$ which is Koszul
up to shift, and the exact sequence $0\rightarrow M_{\geq s}\rightarrow
M\rightarrow M/M_{\geq s}\rightarrow 0$, which induces an exact sequence of
complexes $0\rightarrow \Phi (M/M_{\geq s})\rightarrow \Phi (M)\rightarrow
\Phi (M_{\geq s})\rightarrow 0$ where $\Phi (M/M_{\geq s})$ is a finite
complex and $\Phi (M_{\geq s})$ is exact, except at minimal degree, it
follows by the long homology sequence that $H^{i}(\Phi (M)=0$ for almost all 
$i$.
\end{proof}

We remarked above that the categories $D^{b}(gr_{\Gamma })$ and $%
K^{-,b}(grP_{\Gamma })$ are equivalent as triangulated categories, we have
proved that the image of $\Phi $ is contained in $K^{-,b}(grP_{\Gamma })$.
Composing with the equivalence, we obtain a functor $\Phi ^{\prime }:$ $%
gr_{\Lambda }\rightarrow D^{b}(gr_{\Gamma })$.

Let $\mathcal{A}$ be an abelian category, a Serre subcategory $\mathcal{T}$
of $\mathcal{A}$ is a full subcategory with the property that for every
short exact sequence of $\mathcal{A}$, say, $0\rightarrow A\rightarrow
B\rightarrow C\rightarrow 0$ the object $B$ is in $\mathcal{T}$ if and only
if $A$, $C\in \mathcal{T}$. By [6], we have a quotient abelian category $%
\mathcal{A}/\mathcal{T}$ and an exact functor $\pi :\mathcal{A}$ $%
\rightarrow \mathcal{A}/\mathcal{T}$, which induces at the level of derived
categories an exact functor: $D(\pi ):D(\mathcal{A)}$ $\rightarrow D(%
\mathcal{A}/\mathcal{T)}$. The following result is well known:

\begin{lemma}
\lbrack 20] The kernel of $D(\pi )$ is the full subcategory $\mathcal{K}$
with objects the complex with homology in $\mathcal{T}$ and $D(\pi )$
induces an equivalence of categories $D^{\ast }(\mathcal{A)}$ $/\mathcal{K}%
\cong D^{\ast }(\mathcal{A}/\mathcal{T)}$ for $\ast =+,-,b$.
\end{lemma}

We apply the lemma in the following situation:

Let $\Gamma $ be a noetherian Koszul algebra, $gr_{\Gamma }$ the category of
finitely generated graded $\Gamma $-modules. Let $Qgr_{\Gamma }$ be the
quotient category of $gr_{\Gamma }$ by the Serre subcategory of the modules
of finite length. Let $\pi :gr_{\Gamma }\rightarrow Qgr_{\Gamma }$ be the
natural projection and $D(\pi ):D^{b}(gr_{\Gamma })\rightarrow
D^{b}(Qgr_{\Gamma })$ the induced functor. Denote by $\mathcal{F}_{\Gamma }$
be the full subcategory of $D^{b}(gr_{\Gamma })$ consisting of bounded
complexes of graded $\Gamma $-modules of finite length. Then we have:

\begin{theorem}
\lbrack 16]The functor $D(\pi )$ : $D^{b}(gr_{\Gamma })\rightarrow
D^{b}(Qgr_{\Gamma })$ has kernel $\mathcal{F}_{\Gamma }.$ It induces an
equivalence of triangulated categories $\sigma :D^{b}(gr_{\Gamma })$ /$%
\mathcal{F}_{\Gamma }\rightarrow D^{b}(Qgr_{\Gamma })$.
\end{theorem}

Let $q:D^{b}(gr_{\Gamma })$ $\rightarrow D^{b}(gr_{\Gamma })$ /$\mathcal{F}%
_{\Gamma }$ be the quotient functor. Then $\sigma q=D(\pi )$. The functor $%
j:K^{-,b}(grP_{\Gamma })\rightarrow D^{b}(gr_{\Gamma })$ is truncation, $j$
is an equivalence.

Let $\Lambda $ be a Koszul algebra with Yoneda algebra $\Gamma $ such that
both are AS graded Gorenstein noetherian algebras of finite local cohomology
dimension on both sides. The functor $\theta :gr_{\Lambda }\rightarrow
D^{b}(Qgr_{\Gamma })$ is the composition: $gr_{\Lambda }\overset{\Phi }{%
\rightarrow }\mathfrak{lcp}_{\Gamma }^{-,b}\overset{i}{\rightarrow }%
K^{-,b}(grP_{\Gamma })\overset{j}{\rightarrow }D^{b}(gr_{\Gamma })$ $\overset%
{D(\pi )}{\rightarrow }D^{b}(Qgr_{\Gamma })$ , where $i$ is just the
inclusion.

\begin{center}
$%
\begin{array}{ccccccc}
\mathfrak{lcp}_{\Gamma }^{-,b} & \overset{i}{\rightarrow } & \text{K}^{-,b}%
\text{(grP}_{\Gamma }\text{)} & \overset{j}{\rightarrow } & \text{D}^{b}%
\text{(gr}_{\Gamma }\text{)} & \overset{q}{\longrightarrow } & \text{D}^{b}%
\text{(gr}_{\Gamma }\text{)/}\mathcal{F}_{\Gamma } \\ 
\Phi \uparrow &  &  &  & \text{D(}\pi \text{)}\downarrow & \sigma \swarrow & 
\\ 
\text{gr}_{\Lambda } &  & \overset{\theta }{\rightarrow } &  & \text{D}^{b}%
\text{(Qgr}_{\Gamma }\text{)} &  & 
\end{array}%
$
\end{center}

Now let $P$ be a finitely generated projective graded $\Lambda $-module, $%
P=\oplus P_{i}[n_{i}]$, with each $P_{i}$ generated in degree zero. Then $%
\Phi (P)$ is isomorphic in the category of complexes over $gr_{\Gamma }$ to $%
\oplus \Phi (P_{i})[n_{i}]$ and each $\Phi (P_{i})$ is a projective
resolution of a semisimple $\Gamma $-module. It follows $\theta $ sends any
map factoring through a graded projective module to a zero map in D$^{b}$(Qgr%
$_{\Gamma }$). Consequently, $\theta $ induces a functor \underline{$\theta $%
}:\underline{$gr$}$_{\Gamma }\rightarrow $D$^{b}$(Qgr$_{\Gamma }$). The
functor $\theta $ sends exact sequences to exact triangles, the syzygy
functor $\Omega :$\underline{$gr$}$_{\Lambda }\rightarrow $\underline{$gr$}$%
_{\Lambda }$ is an endofunctor that makes \underline{$gr$}$_{\Lambda }$
"half" triangulated, given an exact sequence $0\rightarrow A\overset{j}{%
\rightarrow }B\overset{t}{\rightarrow }C\rightarrow 0$ in $gr_{\Lambda }$
and $p:P\rightarrow C$ the projective cover, there is an induced exact
commutative diagram:

\begin{center}
$%
\begin{array}{ccccccc}
0\rightarrow & \Omega (C) & \rightarrow & P & \rightarrow & C & \rightarrow 0
\\ 
& w\downarrow &  & \downarrow &  & \downarrow 1 &  \\ 
0\rightarrow & A & \rightarrow & B & \rightarrow & C & \rightarrow 0%
\end{array}%
$
\end{center}

We obtain a half triangle: $\Omega (C)\rightarrow A\rightarrow B\rightarrow
C $ and \underline{$\theta $} sends the half triangle into a triangle in D$%
^{b} $(Qgr$_{\Gamma }$). We want to construct a triangulated category 
\underline{$gr$}$_{\Lambda }[\Omega ^{-1}]$ such that $\Omega $ is an
equivalence which acts as the shift and a functor of half triangulated
categories $\lambda :$\underline{$gr$}$_{\Lambda }\rightarrow $\underline{$%
gr $}$_{\Lambda }[\Omega ^{-1}]$ such that given any triangulated category $%
D $ and a functor of half triangulated categories: $\beta $ \underline{$gr$}$%
_{\Lambda }\rightarrow D$ there is a unique functor of triangulated
categories $\overset{\wedge }{\beta }:$ \underline{$gr$}$_{\Lambda }[\Omega
^{-1}]\rightarrow D$ such that $\overset{\wedge }{\beta }\lambda =\beta $.

We recall the construction given by Buchweitz and reproduced in [2], [15].

Let $(\mathcal{A}$, $\phi )$ be a category with endofunctor, if $(\mathcal{B}
$,$\psi )$ is another pair, then a functor $F:\mathcal{A}\rightarrow 
\mathcal{B}$ is said a morphism of pairs if it makes the diagram

\begin{center}
$%
\begin{array}{ccc}
\mathcal{A} & \overset{\phi }{\rightarrow } & \mathcal{A} \\ 
\downarrow F &  & \downarrow F \\ 
\mathcal{B} & \overset{\psi }{\rightarrow } & \mathcal{B}%
\end{array}%
$
\end{center}

commute, this is: the functors $F\phi $ and $\psi F$ are naturally
isomorphic. If $\psi $ happens to be an auto equivalence, we say that the
morphism $F$ inverts $\phi .$Then there is a the following universal
problem. Given a pair $(\mathcal{A}$, $\phi )$, find a pair $(\mathcal{A}$[$%
\phi ^{-1}],\rho )$ and a morphism of pairs $G:$ $(\mathcal{A}$, $\phi
)\rightarrow $ $(\mathcal{A}$[$\phi ^{-1}],\rho )$ such that $G$ inverts $%
\phi $ and for any morphism of pairs $F:$ $(\mathcal{A}$, $\phi )$ $%
\rightarrow $ $(\mathcal{B}$,$\psi )$ such that $F$ inverts $\phi $, there
is a unique morphism of pairs $F^{\prime }:(\mathcal{A}$[$\phi ^{-1}],\rho
)\rightarrow (\mathcal{B}$,$\psi )$ making the diagram

\begin{center}
$%
\begin{array}{cccccc}
(\mathcal{A},\phi ) &  & \overset{F}{\rightarrow } &  & (\mathcal{B},\psi )
&  \\ 
G & \searrow &  & \nearrow & F^{\prime } &  \\ 
&  & (\mathcal{A}[\phi ^{-1}],\rho ) &  &  & 
\end{array}%
$
\end{center}

Commute.

The objects of $\mathcal{A}$[$\phi ^{-1}]$ are the formal symbols $\phi
^{-n}M$ where $M$ is an object of $\mathcal{A}$ and $n\geq 0$, $\phi ^{0}M=M$%
. If $M$, $N$ are objects in $\mathcal{A}$[$\phi ^{-1}],$we define the
morphisms by

\begin{center}
$%
\begin{array}{c}
Mor_{\mathcal{A}[\phi ^{-1}]}(M,N)=\underset{k}{\underrightarrow{\lim }}Mor_{%
\mathcal{A}}(\phi ^{k}M,\phi ^{k}N)%
\end{array}%
$
\end{center}

where we assume $M=\phi ^{-m}M^{\prime }$ and $N=\phi ^{-n}N^{\prime }$ and $%
k\geq $max\{$m,n$\} .(See [MM] for details)

We define the endofunctor $\rho :\mathcal{A}[\phi ^{-1}]\rightarrow \mathcal{%
A}[\phi ^{-1}]$ by setting $\rho (M)=\phi (M)$ and $\rho (\phi ^{-n}M)=\phi
^{-n+1}(M)$ for any $M$ in $\mathcal{A}$ and any natural number $n$. If $f$
is a morphism represented by some $f_{n}:\phi ^{n}M\rightarrow \phi ^{n}N$
and $n$ sufficiently large, then $\rho (f)$ is represented by $\phi (f_{n})$.

We obtain the morphism of pairs $G:$ $(\mathcal{A}$, $\phi )\rightarrow $ $(%
\mathcal{A}$[$\phi ^{-1}],\rho )$ having the desired properties.

We apply this construction to our pair (\underline{$gr$}$_{\Lambda },\Omega
) $ to obtain a pair (\underline{$gr$}$_{\Lambda }[\Omega ^{-1}]$, $\Omega )$
and a map of pairs $G:$ (\underline{$gr$}$_{\Lambda },\Omega )$ $\rightarrow 
$ (\underline{$gr$}$_{\Lambda }[\Omega ^{-1}]$, $\Omega ^{-1})$

One can check as in [15] or [2] that (\underline{$gr$}$_{\Lambda }[\Omega
^{-1}]$, $\Omega ^{-1})$ is a triangulated category and $\underline{\theta }%
: $\underline{$gr$}$_{\Lambda }\rightarrow $D$^{b}$(gr$_{\Gamma }$) induces
an exact functor $\overset{\wedge }{\theta }:$\underline{ $gr$}$_{\Lambda
}[\Omega ^{-1}]\rightarrow $D$^{b}$(Qgr$_{\Gamma }$) such that the triangle

\begin{center}
$%
\begin{array}{ccc}
\underline{gr}_{\Lambda } &  &  \\ 
\lambda \downarrow & \searrow \underline{\theta } &  \\ 
\underline{\ gr}_{\Lambda }[\Omega ^{-1}] & \overset{\overset{\wedge }{%
\theta }}{\longrightarrow } & D^{b}(Qgr_{\Gamma })%
\end{array}%
$
\end{center}

We now state the main result of the paper.

\begin{theorem}
Let $\Lambda $ be a Koszul algebra with Yoneda algebra $\Gamma $ such that
both are AS graded Gorenstein noetherian algebras of finite local cohomology
dimension on both sides. Then the linearization functor

$\overset{\wedge }{\theta }:\underline{\ gr}_{\Lambda }[\Omega
^{-1}]\rightarrow D^{b}(Qgr_{\Gamma })$ is a duality of triangulated
categories.
\end{theorem}

\begin{proof}
We will only check the functor $\overset{\wedge }{\theta }$ is dense, for
the rest of the proof we proceed as in [15].

Choose any bounded complex $B_{\bullet }$ of finitely generated graded $%
\Gamma $-modules. By Proposition 8, the complex $B_{\bullet }$ is isomorphic
in $D^{b}(Qgr_{\Gamma })$ to a totally linear complex, which is in turn, by
Proposition 7, isomorphic to a linear complex $P_{\bullet }$ of finitely
generated graded projective $\Gamma $-modules with zero homology except for
a finite number of indices. By Proposition 6, there is a finitely generated
graded $\Lambda $-module $M$ such that $\Phi (M)\cong P_{\bullet }$.
Therefore : $\overset{\wedge }{\theta }(M)\cong B_{\bullet }$ in $%
D^{b}(Qgr_{\Gamma })$.
\end{proof}

\begin{corollary}
Let $\Lambda $ be a Koszul algebra with Yoneda algebra $\Gamma $ such that
both are AS graded Gorenstein noetherian algebras of finite local cohomology
dimension on both sides. Then the linearization functor $\overset{\wedge }{%
\theta ^{\prime }}:\underline{\ gr}_{\Gamma }[\Omega ^{-1}]\rightarrow
D^{b}(Qgr_{\Lambda })$ is a duality of triangulated categories.
\end{corollary}

\begin{proof}
It follows by symmetry.
\end{proof}

\begin{center}
{\LARGE References}
\end{center}

\bigskip

[1] Artin M., Schelter W. Graded algebras of global dimension 3, Adv. Math.
66 (1987), 171-216.

[2] Beligiannis A. The homological theory of contravariantly finite
subcategories:\ Auslander-Buchweitz Contexts, Gorenstein Categories and
(C9)-Stabilizations, Comm. in Algebra 28 (19), (2000), 4547-4596.

[3] Bernstein J., Gelfand I.M., Gelfand S.I., Algebraic vector bundles over P%
$^{n}$ and problems of linear algebra. Finkt. Anal. Prilozh. 12, No. 3,
66-67, (1978) English transl. Funct. Anal. Appl. 12, 212-214 (1979)

[4] Carlson J.F. The Varieties and the Cohomology Ring of a Module, J. of
Algebra, Vol. 85, No. 1, (1983) 104-143.

[5] Cartan H. Eilenberg S. Homological Algebra, Princeton Mathematical
Series 19, Princeton University Press 1956.

[6] Gabriel P. Des Cat\'{e}gories abeliennes, Bull. Soc. Math. France, 90
(1962), 323-448.

[7] Gelfand S.I., Manin Yu. I., Methods of homological algebra,
Springer-Verlag (1996).

[8] J$\varnothing $rgensen P. Local Cohomology for Non Commutative Graded
Algebras. Comm. in Algebra, 25(2), 575-591 (1997)

[9] J$\varnothing $rgensen P. Linear free resolutions over non-commutative
algebras. Compositio Math. 140 (2004) 1053-1058.

[10] J\o rgensen, P.; Zhang, James J. Gourmet's guide to Gorensteinness.
Adv. Math. 151 (2000), no. 2, 313--345.

[11] Mart\'{\i}nez-Villa, R. Applications of Koszul algebras: the
preprojective algebra. Representation theory of algebras (Cocoyoc, 1994),
487--504, CMS Conf. Proc., 18, Amer. Math. Soc., Providence, RI, 1996.

[12] Martinez-Villa, R. Graded, Selfinjective, and Koszul Algebras, J.
Algebra 215, 34-72 1999

[13] Martinez-Villa, R. Koszul algebras and the Gorenstein condition.
Representations of algebras (S\~{a}o Paulo, 1999), 135--156, Lecture Notes
in Pure and Appl. Math., 224, Dekker, New York, 2002.

[14] Martinez-Villa, R. Local cohomology and non commutative Gorenstein
Algebras, Preprint, Centro de Ciencias Matem\'{a}ticas, UNAM (2012).

[15] Mart\'{\i}nez-Villa, R., Martsinkovsky, A. Stable Projective Homotopy
Theory of Modules, Tails, and Koszul Duality, Comm. Algebra 38 (2010), no.
10, 3941--3973.

[16] Mart\'{\i}nez Villa, R; Saor\'{\i}n, M. Koszul equivalences and
dualities. Pacific J. Math. 214 (2004), no. 2, 359--378.

[17] Mart\'{\i}nez-Villa, R.; Zacharia, Dan Approximations with modules
having linear resolutions. J. Algebra 266 (2003), no. 2, 671--697.

[18] Mart\'{\i}nez-Villa, R. ; Zacharia, Dan Selfinjective Koszul algebras.
Th\'{e}ories d'homologie, repr\'{e}sentations et alg\`{e}bres de Hopf. AMA
Algebra Montp. Announc. 2003, Paper 5, 5 pp. (electronic).

[19] Miyachi, Jun-Ichi, Derived Categories with Applications to
Representation of Algebras, Chiba University, June 2000.

[20] Miyachi, Jun-Ichi, Localization of triangulated categories and derived
categories, J. Algebra, 141 (1991), 463-483.

[21] Mazorchuk, V. Ovsienko, S. A pairing in homology and the category of
linear complexes of tilting modules for a quasi-hereditary algebra, J. Math.
Kyoto Univ. 45 (2005) no. 4, 711-741.

[22] Mori, I., Rationality of the Poincare series for Koszul algebras,
Journal of Algebra, V. 276, no. 2, (2004) pag. 602-624.

[23] Popescu N. Abelian categories with applications to rings and modules,
Academic Press (1973).

[24] Rotman J.J. An Introduction to Homological Algebra, Second Edition
Universitext, Springer, 2009.

[25] Smith P. Some finite dimensional algebras related to elliptic curves,
"Rep. Theory of Algebras and Related Topics", CMS Conference Proceedings,
Vol. 19, 315-348, Amer. Math. Soc. Providence, 1996.

\end{document}